\documentclass[10pt,draft,reqno]{amsart}
     \makeatletter
     \def\section{\@startsection{section}{1}%
     \z@{.7\linespacing\@plus\linespacing}{.5\linespacing}%
     {\bfseries
     \centering
     }}
     \def\@secnumfont{\bfseries}
     \makeatother
\newtheorem{theorem}{Theorem}[section]
\newtheorem{lemma}[theorem]{Lemma}
\newtheorem{proposition}[theorem]{Proposition}

\theoremstyle{definition}
\newtheorem{definition}[theorem]{Definition}
\newtheorem{example}[theorem]{Example}
\theoremstyle{remark}
\newtheorem{remark}[theorem]{Remark}
\theoremstyle{definition}
\newtheorem{assumption}[theorem]{Assumption}

\usepackage{amsmath,amsfonts}
\usepackage{todonotes}
\usepackage{mathtools}
\usepackage{mleftright}			
\usepackage{latexsym}
\usepackage{tikz,tikz-cd}
\usepackage{comment}
\usepackage{tikzit}				
	\usetikzlibrary{decorations.pathreplacing}

\tikzstyle{morphism}=[fill=white, draw=black, shape=rectangle]
\tikzstyle{medium box}=[fill=white, draw=black, shape=rectangle, minimum width=0.7cm, minimum height=0.7cm]
\tikzstyle{large morphism}=[fill=white, draw=black, shape=rectangle, minimum width=1.7cm, minimum height=1cm]
\tikzstyle{bn}=[fill=black, draw=black, shape=circle, inner sep=1.5pt]
\tikzstyle{state}=[fill=white, draw=black, regular polygon, regular polygon sides=3, minimum width=0.8cm, shape border rotate=180, inner sep=0pt]
\tikzstyle{medium state}=[fill=white, draw=black, regular polygon, regular polygon sides=3, minimum width=1.3cm, inner sep=0pt, shape border rotate=180]
\tikzstyle{large state}=[fill=white, draw=black, regular polygon, regular polygon sides=3, minimum width=2.2cm, shape border rotate=180, inner sep=0pt]
\tikzstyle{wide state}=[fill=white, draw=black, shape=isosceles triangle, minimum width=0.8cm, shape border rotate=270, inner sep=1.4pt, minimum height=0.5cm, isosceles triangle apex angle=80]
\tikzstyle{wn}=[fill=white, draw=black, shape=circle, inner sep=1.5pt]
\tikzstyle{blue morphism}=[fill=white, draw={rgb,255: red,15; green,0; blue,150}, shape=rectangle, text={rgb,255: red,15; green,0; blue,150}, tikzit category=blue]
\tikzstyle{blue state}=[fill=white, draw={rgb,255: red,15; green,0; blue,150}, shape=circle, regular polygon, regular polygon sides=3, minimum width=0.8cm, shape border rotate=180, inner sep=0pt, text={rgb,255: red,15; green,0; blue,150}, tikzit category=blue]
\tikzstyle{blue node}=[fill={rgb,255: red,15; green,0; blue,150}, draw={rgb,255: red,15; green,0; blue,150}, shape=circle, tikzit category=blue, inner sep=1.5pt]
\tikzstyle{blue}=[text={rgb,255: red,15; green,0; blue,150}, tikzit draw={rgb,255: red,191; green,191; blue,191}, tikzit category=blue, tikzit fill=white, inner sep=0mm]
\tikzstyle{blue wide state}=[fill=white, draw={rgb,255: red,15; green,0; blue,150}, text={rgb,255: red,15; green,0; blue,150}, shape=isosceles triangle, minimum width=0.8cm, shape border rotate=270, inner sep=1.4pt, minimum height=0.5cm, isosceles triangle apex angle=80]
\tikzstyle{red node}=[fill={rgb,255: red,150; green,0; blue,2}, draw={rgb,255: red,150; green,0; blue,2}, shape=circle, inner sep=1.5pt]
\tikzstyle{Purple node}=[fill={rgb,255: red,150; green,0; blue,150}, draw={rgb,255: red,150; green,0; blue,150}, shape=circle, inner sep=1.5pt]
\tikzstyle{red}=[text={rgb,255: red,150; green,0; blue,2}, inner sep=0mm, tikzit fill=white, tikzit draw={rgb,255: red,191; green,191; blue,191}]
\tikzstyle{purple}=[text={rgb,255: red,150; green,0; blue,150}, inner sep=0mm, tikzit fill=white, tikzit draw={rgb,255: red,191; green,191; blue,191}]
\tikzstyle{white morphism}=[fill=white, draw=white, shape=rectangle, tikzit draw={rgb,255: red,139; green,139; blue,139}]

\tikzstyle{arrow}=[->]
\tikzstyle{dashed box}=[-, dashed]
\tikzstyle{blue arrow}=[-, draw={rgb,255: red,15; green,0; blue,150}, tikzit category=blue]
\tikzstyle{mapsto}=[{|->}]
\tikzstyle{double wire}=[-, double]
\tikzstyle{curly brace}=[decorate,decoration={brace,amplitude=5pt}]
\usepackage[T1,T2A]{fontenc}			
\usepackage[pagebackref,draft=false]{hyperref}

\newcommand{\N}{\mathbb{N}}

\newcommand{\ph}{\mathord{\rule[-0.05em]{0.6em}{0.05em}}}		
\newcommand{\smalltitle}[1]{%
	\par\bigskip\noindent%
	\paragraph{\textbf{#1}}%
}

\newcommand{\cat}[1]{{\mathsf{#1}}} 

\newcommand{\id}{\mathrm{id}} 		

\newcommand{\R}{\mathbb{R}}
\newcommand{\mydots}{\cdotp\!\cdotp\!\cdotp}
\newcommand{\tailcond}[1]{{#1}_{|\mathrm{tail}}}

\newcommand{\tensor}{\otimes}

\tikzset{pullback/.style={minimum size=1.2ex,path picture={	
			\draw[opacity=1,black,-,#1] (-0.5ex,-0.5ex) -- (0.5ex,-0.5ex) -- (0.5ex,0.5ex);%
}}}

\newcommand{\cC}{\mathsf{C}}		
\renewcommand{\det}{\mathrm{det}}	
\newcommand{\samp}{\mathsf{samp}}	

\newcommand{\FinStoch}{\mathsf{FinStoch}}

\newcommand{\BorelStoch}{\mathsf{BorelStoch}}
\newcommand{\as}[1]{
		\def\relstate{#1}%
		\ifx\relstate\empty
		  \text{a.s.}%
		\else
		  {#1\text{-a.s.}}%
		\fi
	}
\newcommand{\ase}[1]{=_{#1\text{-a.s.}}}					

\DeclareMathOperator{\cop}{copy}
\DeclareMathOperator{\discard}{del}

	\DeclarePairedDelimiterX{\Set}[1]{\{}{\}}{%
		
		#1
	}
	\makeatletter
		\let\oldSet\Set
		\def\Set{\@ifstar{\oldSet}{\oldSet*}}
	\makeatother

	\DeclarePairedDelimiterX{\Family}[1]{(}{)}{%
		
		#1
	}
	\makeatletter
		\let\oldFamily\Family
		\def\Family{\@ifstar{\oldFamily}{\oldFamily*}}

	\makeatletter
		\def\@textbottom{\vskip \z@ \@plus 18pt}
		\let\@texttop\relax
	\makeatother

\begin{document}

\title[De Finetti's Theorem in Categorical Probability]{De Finetti's Theorem in Categorical Probability}

\author{Tobias Fritz*}
\thanks{* Corresponding author}
\address{Tobias Fritz: Department of Mathematics, University of Innsbruck, Austria}
\email{tobias.fritz@uibk.ac.at}
\urladdr{http://tobiasfritz.science/}

\author{Tom{\'a}{\v{s}} Gonda}
\address{Tom{\'a}{\v{s}} Gonda: Perimeter Institute for Theoretical Physics; School of Physics and Astronomy, University of Waterloo, Waterloo ON, Canada}
\email{tgonda@pitp.ca}
\urladdr{https://perimeterinstitute.ca/people/tomas-gonda}

\author{Paolo Perrone}
\address{Paolo Perrone: Department of Computer Science, University of Oxford, United Kingdom}
\email{paolo.perrone@cs.ox.ac.uk}
\urladdr{http://www.paoloperrone.org}



\begin{abstract}
    We present a novel proof of de Finetti's Theorem characterizing permutation-invariant probability measures of infinite sequences of variables, so-called exchangeable measures.
	The proof is phrased in the language of Markov categories, which provide an abstract categorical framework for probability and information flow. 
	The diagrammatic and abstract nature of the arguments makes the proof intuitive and easy to follow.
	We also show how the usual measure-theoretic version of de Finetti's Theorem for standard Borel spaces is an instance of this result.
\end{abstract}

\maketitle

\section{Introduction}

	De Finetti's Theorem states that every permutation-invariant joint probability distribution of countably many random variables is such that these variables are conditionally independent given a suitable latent variable.
	Moreover, a canonical choice of the latent variable is the empirical distribution of the variables under consideration (which exists almost surely). 
	In this paper, we state and prove a more abstract version of de Finetti's Theorem in the context of \emph{categorical probability theory}, which is a nascent framework for a foundation of probability and statistics that is more abstract and more general than the traditional measure-theoretic approach.

\smalltitle{Context of de Finetti's Theorem.}

	To provide some context for and illustrate the significance of de Finetti's Theorem, it is helpful to consider the notoriously divisive debate on the subjective vs.~objective view on probability \cite{sep-probability-interpret,suppes2002representation}.
	For our purposes, note that there are situations in which it is commonly presumed that one can give objective meaning to probability and infer its value.
	For example, flipping a coin and counting the relative frequency of `heads' among the trials ought to converge to the bias of the coin.
	This is the content of the law of large numbers.
	
	However, there is no free lunch and the assumption that the objective probability \emph{exists} is crucial in the above reasoning.
	As de Finetti eloquently argues in Section 3 of \cite{de1937prevision}, from the subjectivist point of view,
	\begin{quote}
		``nothing obliges us to choose [probability of heads in the next toss] to be close to the frequency [of heads in the previously observed tosses].''
	\end{quote}
	Let us give a brief account of the argument here and refer to de Finetti's classic \cite{de1937prevision} for more details.
	
	In the coin-flipping experiment, suppose that we are given a (subjective) joint probability distribution over possible sequences of heads and tails.
	If the coin had an associated bias that does not change from one trial to the next, we would expect the outcomes of different trials to be conditionally independent, given the knowledge of the bias.
	In such case, however, the bias is an (objective) property of the coin that corresponds to the probabilities that the coin lands heads or tails respectively.
	On the flip side, it is unclear how to justify trial independence without implicitly bringing in the assumption of objective probability. 
	Nevertheless, we may require that the joint distribution is unchanged under permutations of the different trials, a property known as \emph{exchangeability}\footnotemark{} which expresses the belief that these are trials of the ``same phenomenon''.
	\footnotetext{Exchangeable measures are occasionally also called ``symmetric'' or ``equivalent'', see the translator's note of \cite{de1937prevision}.}

	De Finetti shows that with respect to any exchangeable probability measure, the individual trials are conditionally independent given some random variable $Y$.
	Moreover, there is a canonical choice for $Y$ whose values precisely correspond to the possible biases of the coin, so that (in this example) $Y$ takes values in the unit interval $[0,1]$.
	De Finetti's Theorem thus justifies the use of probability distributions in the subjectivist conception, since it says that exchangeable behaviors (of infinitely many trials) are indistinguishable from the behaviors of a repeated flipping of a coin with unknown, but fixed, bias.
	
	Indeed, one can easily see that flipping a coin with an unknown bias, distributed according to some probability measure on $Y$, leads to a distribution of sequences of heads and tails that is exchangeable.
	The non-trivial part of de Finetti's Theorem is in the converse.
	It says that, given a countable number of trials, \emph{every} exchangeable distribution arises in this way.
	
	Historically speaking, the above context is the one in which the topic of our paper---de Finetti's Theorem---originates.
	However, besides its implications for the foundations of the concept of probability, the result also bears significance for the development of nonparametric Bayesian modelling \cite{jordan2010bayesian,sep-statistics}.
	More concretely, many natural stochastic processes---such as drawing with replacement (e.g.\ the Pólya urn model and modifications thereof) or without replacement---give rise to exchangeable distributions.
	By de Finetti's Theorem, they can be equivalently described as drawing a random probability distribution $q$ according to a specified prior \emph{distribution over distributions}, $\mu$ (e.g.\ a Dirichlet process), and subsequently drawing independent samples from $q$.
	In other words, they correspond to a mixture of iid (independent and identically distributed) samplings, which is useful both as an aid for intuitive understanding and for concrete calculations.
	
\smalltitle{Existing proofs and variants of de Finetti's Theorem.}

	The original proofs by de Finetti in the binary case (as in coin-flipping) can be found in \cite{de1929funzione,de1937prevision}, while a more general result is due to Hewitt and Savage \cite{hewitt1955symmetric}, who have shown the analogous statement for exchangeable Radon probability measures on compact Hausdorff spaces.
	We give more details on the measure-theoretic formulation for standard Borel spaces in Section~\ref{sec:classical}.
	Since then, many proofs based on different methods have appeared in various contexts.
	For instance, those of \cite[Theorem 11.10]{kallenberg2002foundations} and \cite[Theorem 3.1]{austin2008exchangeable} use the mean ergodic theorem and conditional expectations; in \cite[Theorem~12.17]{klenke2014probability} and~\cite[Theorem~3.1]{aldous1985exchangeability} one arrives at the result via backwards martingales; while harmonic \cite{ressel1985analytical} and non-standard \cite{alam2020nonstandard} analysis have also been used.
	The latter has recently led to a new generalization of de Finetti's Theorem for exchangeable Radon measures on \emph{any} Hausdorff space~\cite{alam2020generalizing}.
	Functional analysis and moment methods have been utilized in the proof from \cite[Chapter VII]{feller2008introduction}, whose more elementary and explicitly calculational version that applies to binary variables can be found in~\cite{kirsch2019elementary}.
	
	Very recently, Jacobs and Staton \cite{jacobsstaton2020definetti} have presented a category-theoretic approach to de Finetti's Theorem, different from the present one, for the binary case. 
	Therein, the statement of de Finetti's Theorem is encoded in the fact that the unit interval---the space of probability distributions---arises as the categorical limit of a sequence of multisets over $\{0,1\}$ (``urns'') related by morphisms that represent the action of drawing a random element from an urn without replacement. 
	
\smalltitle{Our version of de Finetti's Theorem.}

	In Section~\ref{sec:synthetic}, we present an abstract version of de Finetti's Theorem as a statement about morphisms in categories which admit notions of parallel composition, swapping, copying, and discarding, called \emph{Markov categories} \cite{fritz2019synthetic}.
	The framework of Markov categories is very general, and there are many Markov categories in which de Finetti's Theorem does not hold.
	Correspondingly, there are additional axioms that enter as ingredients in the proof of our abstract version of de Finetti's Theorem.
	Specifically, they correspond to the fact that one should be able to 
	\begin{itemize}
		\item construct conditional probabilities, 
		\item describe distributions on spaces of probability distributions, and
		\item consider countable sequences of trials.
	\end{itemize}
	These three requirements will be stated formally in Assumption~\ref{three_assumptions}, after introducing the basics of the formalism of Markov categories in Section~\ref{sec:Markov_cat}.
	Within a Markov category that satisfies them, one therefore obtains a version of de Finetti's Theorem.
	As we show in Section~\ref{sec:synthetic}, the measure-theoretic version for discrete and continuous random variables (described by standard Borel spaces) arises in this way as well.
	Additionally, we automatically get a characterization of exchangeable Markov kernels, unlike the classical results that focus exclusively on exchangeable measures.
	As far as we know, this result has not appeared in the literature before.
	
	It is likely that other Markov categories also have the required properties, and that instantiating the abstract de Finetti's Theorem in those leads to a context in which the result would be entirely new.
	However, we do not know of any such examples at present and leave the search for these to future investigations.
	
\smalltitle{Synthetic probability theory.}

	The process described above---identifying abstract results that follow from properties expressed within the framework and instantiating them in concrete Markov categories---is an example of the \emph{synthetic} approach to probability theory,
	which is distinguished from the standard \emph{analytic} one in terms of measure theory by encapsulating measure-theoretic statements in suitable higher-level axioms \cite{fritz2019synthetic}.
	It differs from standard approaches by using more formal and abstract reasoning and it only depends on the particular measure-theoretic semantics insofar as the synthetic axioms may or may not be satisfied.
	A number of concepts and theorems of classical probability and statistics have been given a synthetic treatment in recent years \cite{chojacobs2019strings,fritz2019synthetic,fritzrischel2019zeroone,fritz2020representable}.
	We recall the ones relevant for de Finetti's Theorem in Section~\ref{sec:Markov_cat}.
	
	The proof that we present in Section~\ref{sec:proof} is inspired by several of the concrete ones mentioned above, but as far as we know it does not match either of them completely.
	Its abstract nature, which places the focus on essential aspects while allowing us to ignore irrelevant details, arguably makes it easier to follow than any measure-theoretic one, although it is still far from obvious.
	The proof relies on the graphical calculus of string diagrams, which captures several non-trivial properties implicitly and whose connectivity explicitly depicts information flow.
	It is these two features that enable one to readily interpret any stage of the proof with relative ease, once some familiarity with the diagrammatic notation has been obtained.

\smalltitle{Outlook.}

	Given the intuitive nature of our proof, it is natural to hope that even deeper results can be proven along similar lines in a purely synthetic manner, and that one can ultimately aim at proving \emph{new} statements that would be too difficult to obtain in the traditional measure-theoretic formalism due to its higher complexity. 
	With this in mind, it may be worth mentioning a few extensions of de Finetti's Theorem which one can try to aim at next.
	
	Among the most interesting variations on the exchangeability theme is arguably the notion of \emph{partial exchangeability}~\cite{diaconis1980finetti}, where the invariance under finite permutation invariance is relaxed to invariance under certain structure-preserving permutations. 
	For example, the Aldous--Hoover Theorem~\cite{aldous1985exchangeability,austin2008exchangeable} characterizes exchangeable arrays of random variables and is closely related to random graphs.
	There is a similar result for Markov chains~\cite{de1938condition,diaconis1980finettimarkov,camerlenghi2019distribution}. 
	More recently, partial exchangeability has been generalized to hierarchical exchangeability~\cite{austin2014hierarchical,jung2021generalization}. 
	The work of Crane and Towsner provides perhaps the most general currently available results along these lines~\cite{crane2015relatively,crane2018relative}, situated in a model-theoretic framework.

\section{Measure-Theoretic Version of de Finetti's Theorem}
\label{sec:classical}

	Let us turn to a more formal exposition of de Finetti's Theorem in standard measure-theoretic language. 
	Given a measurable space $X$, consider the product $X^\N$ of countably many copies of $X$, equipped with the product $\sigma$-algebra. 
	A bijection $\N\to\N$ which fixes all elements apart from a finite subset is called a \emph{finite permutation}. 
	Given a finite permutation $\sigma \colon \N\to\N$, we denote by $n_\sigma \in \N$ the largest natural number not fixed by $\sigma$.
	
	Consider a probability measure $p$ on $X^\N$. 
	By convention, given a finite collection of measurable subsets $S_1,\dots,S_n\subseteq X$, we write 
	\begin{equation*}		
		p \mleft( S_1 \times \dots \times S_n \mright) 
	\end{equation*}
	as shorthand for the probability of the ``cylinder'' event
	\begin{equation*}
		p \mleft(S_1 \times \dots \times S_n \times X \times X \times \dots  \mright) .
	\end{equation*}
	These probabilities specify the marginal distribution of $p$ on the first $n$ components of $X^\N$. 
	We say that the measure $p$ is \emph{exchangeable} if for every finite permutation $\sigma$ and for every finite sequence of measurable subsets $S_1,\dots,S_{n_\sigma}\subseteq X$, we have
	\begin{equation}
		p \mleft( S_1 \times \dots \times S_{n_\sigma} \mright)  =  p \mleft( S_{\sigma(1)} \times \dots \times S_{\sigma(n_\sigma)} \mright) .
	\end{equation}
	As we saw in the introduction, a somewhat trivial example of exchangeable measures is given by product measures, i.e., ones satisfying
	\begin{equation}
		q \mleft( S_1 \times \dots \times S_n \mright) = q \mleft(S_1 \mright) \times \cdots \times q \mleft(S_n \mright).
	\end{equation}
	Such a product measure $q$ constitutes the law of a sequence of iid random variables.
	
	Having introduced the notion of exchangeability, we now turn to de Finetti's Theorem itself.
	A convenient way to express the statement, which lends itself well to category-theoretical translations, is to use the concept of \emph{measures on a space of measures}, as done by Hewitt and Savage \cite[Section~2]{hewitt1955symmetric} among others. 
	If $X$ is a standard Borel space, we denote by $PX$ the set of probability measures on $X$. 
	The set $PX$ can be equipped with a canonical $\sigma$-algebra, namely the one generated by the functions $\varepsilon_f \colon PX \to \R$ given by
	\begin{equation*}
		p \mapsto \int f(x) \, p ( dx )
	\end{equation*}
	for all bounded measurable functions $f \colon X\to \R$.
	Measures \emph{on} $PX$ (thus elements of $PPX$) can be thought of as \emph{random measures on $X$}, where also the specific form of the distribution is subject to uncertainty.\footnote{Such recursive ways of forming spaces can often be accurately captured by the categorical notion of a monad, and indeed these spaces of measures can be described in terms of a well-known monad, the \emph{Giry monad} \cite{giry}.}
	
	Equivalently, measures on $PX$ describe mixtures of measures---either in the sense of finite convex combinations or integrals. 
	Indeed, de Finetti's Theorem can be summarized as the fact that exchangeable measures are mixtures of product measures. 
	Here is the precise statement, in the version for standard Borel spaces.
	
	\begin{theorem}[de Finetti's Theorem]\label{classicalthm}
		Let $X$ be a standard Borel space. A probability measure $p$ on $X^\N$ is exchangeable if and only if there exists a probability measure $\mu$ on $PX$ such that for every finite collection of measurable subsets $S_1,\dots,S_n\subseteq X$, we have
		\begin{equation}
			p \mleft( S_1 \times \dots \times S_n \mright) = \int_{PX} q(S_1) \times \cdots \times q(S_n) \, \mu(dq) .
		\end{equation}
	\end{theorem}
	For example, the iid case amounts to $\mu$ being a delta measure $\delta_r \in PPX$ for some $r \in PX$ satisfying $p = r^\N$.
	On the other hand, if $\mu$ is a random measure supported on the set $\{\delta_x\}_{x \in X}$ of delta measures $\delta_x \in PX$, then we obtain $p$ supported on the diagonal---i.e., the set of constant sequences in $X^\N$.
	In other words, the associated exchangeable random variables are then perfectly correlated.
	
	By convex analysis arguments, it can also be shown that given $p$ as in Theorem~\ref{classicalthm}, the measure $\mu$ is uniquely determined.
	Note that the result due to Hewitt and Savage~\cite{hewitt1955symmetric} is more general than the one above, as it applies to Radon probability measures on arbitrary compact Hausdorff spaces. 
	It includes the uniqueness of $\mu$ as well.

\section{Markov Categories}\label{sec:Markov_cat}
	
	We now take a detour from the discussion of de Finetti's Theorem into the realm of Markov categories.
	All of the concepts defined in this section have been introduced in earlier works on categorical probability~\cite{chojacobs2019strings,fritz2019synthetic,fritzrischel2019zeroone,fritz2020representable}.
	Nevertheless, in the interest of a self-contained presentation, we recall the main points here in a slightly less formal way, referring to the existing literature for full technical detail.
	
	As mentioned in the introduction, a Markov category is a category that comes with notions of parallel composition, swapping, copying, and discarding.
	The idea is that contexts in which one wants to model flow of information will generally satisfy these basic requirements and thus correspond to a Markov category.
	In particular, one would expect that any formalisation of classical probability theory does.
	However, there are also many Markov categories that have nothing to do with probability theory (see \cite{fritz2019synthetic}).
	
	One of the simplest interesting examples to keep in mind is $\FinStoch$, the category of finite sets and stochastic matrices. 
	An object in $\FinStoch$ is a finite set, which can be interpreted as the set of possible values of a random variable.
	A morphism, say $f \colon A \to X$, assigns a probability measure on the finite set $X$ to each element of the finite set $A$.
	It can thus be described as a \emph{stochastic matrix} with entries $f(x|a)$ indexed by $a \in A$ and $x \in X$.
	Parallel composition of two stochastic maps, $f \colon A \to X$ and $g \colon B \to Y$, is just their tensor product $f \otimes g \colon A \otimes B \to X \otimes Y$, where $A \otimes B$ is the cartesian product of the underlying sets and one multiplies the probabilities in the formation of the tensor product.
	On the other hand, the sequential composition of $f \colon A \to X$ and $h\colon X \to Z$ to produce $h \circ f \colon A \to Z$ is given by the usual matrix multiplication.
	In the context of stochastic maps, it is also known as the \emph{Chapman-Kolmogorov equation}.

	In the string diagrammatic notation that we make heavy use of, objects are depicted as `wires', while morphisms are generally drawn as `boxes', with their domain below and their codomain above the box depicted as incident wires.
	Parallel and sequential composition is depicted by
	\begin{equation*}
		\begin{tikzpicture}
	\begin{pgfonlayer}{nodelayer}
		\node [style=morphism] (0) at (-4.5, 0) {$f$};
		\node [style=none] (1) at (-4.5, -1) {};
		\node [style=none] (2) at (-4.5, 1) {};
		\node [style=none] (3) at (-4.5, 1.5) {$X$};
		\node [style=none] (4) at (-4.5, -1.5) {$A$};
		\node [style=morphism] (5) at (-3, 0) {$g$};
		\node [style=none] (6) at (-3, -1) {};
		\node [style=none] (7) at (-3, 1) {};
		\node [style=none] (8) at (-3, 1.5) {$Y$};
		\node [style=none] (9) at (-3, -1.5) {$B$};
		\node [style=none] (10) at (0, 0) {$\text{and}$};
		\node [style=morphism] (11) at (3, 0.75) {$h$};
		\node [style=morphism] (12) at (3, -0.75) {$f$};
		\node [style=none] (13) at (3, 1.75) {};
		\node [style=none] (14) at (3, 2.25) {$Z$};
		\node [style=none] (15) at (3, -1.75) {};
		\node [style=none] (16) at (3, -2.25) {$A$};
	\end{pgfonlayer}
	\begin{pgfonlayer}{edgelayer}
		\draw (1.center) to (0);
		\draw (0) to (2.center);
		\draw (6.center) to (5);
		\draw (5) to (7.center);
		\draw (15.center) to (12);
		\draw (12) to (11);
		\draw (11) to (13.center);
	\end{pgfonlayer}
\end{tikzpicture}
	\end{equation*}
	respectively.
	
	We depict morphisms $m \colon I\to X$ from the monoidal unit as
	
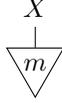
\begin{equation*}
		\begin{tikzpicture}
	\begin{pgfonlayer}{nodelayer}
		\node [style=none] (2) at (0, 1) {};
		\node [style=none] (4) at (0, 1.5) {$X$};
		\node [style=state] (5) at (0, 0) {$m$};
	\end{pgfonlayer}
	\begin{pgfonlayer}{edgelayer}
		\draw (5) to (2.center);
	\end{pgfonlayer}
\end{tikzpicture}

	\end{equation*}
	with the interpretation being that they represent ``random states''. 
	For example, in $\FinStoch$ these are the finitely supported probability measures, i.e.\ probability distributions.
	
	Swapping $X \otimes Y \to Y \otimes X$ is implemented by assigning to each pair $(x,y) \in X \times Y$ the delta distribution $\delta_{(y,x)}$ on the pair $(y,x) \in Y \times X$.
	Diagrammatically, we depict such a swap as
	
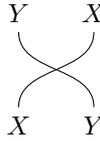
\begin{equation*}
		\begin{tikzpicture}
	\begin{pgfonlayer}{nodelayer}
		\node [style=none] (0) at (-1, -1) {};
		\node [style=none] (1) at (1, -1) {};
		\node [style=none] (2) at (-1, 1) {};
		\node [style=none] (3) at (1, 1) {};
		\node [style=none] (4) at (-1, -1.5) {$X$};
		\node [style=none] (5) at (1, -1.5) {$Y$};
		\node [style=none] (6) at (-1, 1.5) {$Y$};
		\node [style=none] (7) at (1, 1.5) {$X$};
	\end{pgfonlayer}
	\begin{pgfonlayer}{edgelayer}
		\draw [in=-90, out=90] (0.center) to (3.center);
		\draw [in=90, out=-90] (2.center) to (1.center);
	\end{pgfonlayer}
\end{tikzpicture}

	\end{equation*}
	Copying and discarding in $\FinStoch$ are likewise as one would expect.
	Specifically, $\cop_X \colon X \to X \otimes X$ is a morphism that assigns, to each element $x$, the delta distribution $\delta_{(x,x)}$ on the pair $(x,x) \in X \otimes X$.
	Discarding, $\discard_X \colon X \to I$, is the stochastic map that corresponds to marginalization over $X$,
	where $I$ denotes a fixed single-element set.
	Thus to each element of $X$, the stochastic map $\discard_X$ assigns the unique probability measure on $I$.
	Diagrammatically, we represent these maps as
	
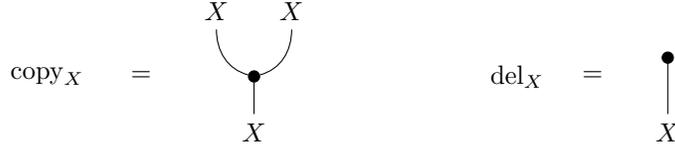
\begin{equation*}
		\begin{tikzpicture}
	\begin{pgfonlayer}{nodelayer}
		\node [style=bn] (0) at (-4, 0) {};
		\node [style=none] (1) at (-5, 1.25) {};
		\node [style=none] (2) at (-4, -1) {};
		\node [style=none] (3) at (-3, 1.25) {};
		\node [style=none] (4) at (-5, 1.75) {$X$};
		\node [style=none] (6) at (-3, 1.75) {$X$};
		\node [style=none] (7) at (-4, -1.5) {$X$};
		\node [style=none] (8) at (-9.5, 0) {$\cop_{X}$};
		\node [style=bn] (9) at (7, 0.5) {};
		\node [style=none] (10) at (7, -1) {};
		\node [style=none] (11) at (7, -1.5) {$X$};
		\node [style=none] (12) at (3, 0) {$\discard_{X}$};
		\node [style=none] (13) at (-7, 0) {$=$};
		\node [style=none] (14) at (5, 0) {$=$};
	\end{pgfonlayer}
	\begin{pgfonlayer}{edgelayer}
		\draw [in=165, out=-90] (1.center) to (0);
		\draw [in=270, out=15] (0) to (3.center);
		\draw (0) to (2.center);
		\draw (9) to (10.center);
	\end{pgfonlayer}
\end{tikzpicture}
	\end{equation*}
	where we implicitly make use of $X \otimes I \cong X \cong I \otimes X$, so that $I$ need not be drawn explicitly.
	
	Abstracting these properties along with corresponding compatibility requirements leads to the notion of Markov category.
	For a more detailed discussion of the definition and its formal aspects, we refer the reader to \cite{fritz2019synthetic}.
	\begin{definition}[Markov category {\cite[Definition 2.1]{fritz2019synthetic}}]
		\label{markov_cat}
		A \emph{Markov category} $\cC$ is a symmetric monoidal category where the monoidal unit object $I$ is a terminal object,\footnote{This means that for every object $X$ in $\cC$ there is a \emph{unique} morphism $X \to I$.} every object $X \in \cC$ is equipped with distinguished morphisms $\cop_X \colon X \to X \otimes X$ and $\discard_{X} \colon X \to I$ that make $X$ into a commutative comonoid, and such that
		\begin{equation}
			\label{multiplicativity}
			\begin{tikzpicture}
	\begin{pgfonlayer}{nodelayer}
		\node [style=none] (0) at (-5, 2.5) {$X\otimes Y$};
		\node [style=none] (1) at (3, -2) {};
		\node [style=bn] (2) at (3, -1) {};
		\node [style=none] (3) at (4, 0) {};
		\node [style=none] (4) at (4, 0) {};
		\node [style=none] (5) at (2, 0) {};
		\node [style=none] (6) at (7, -2) {};
		\node [style=bn] (7) at (7, -1) {};
		\node [style=none] (8) at (8, 0) {};
		\node [style=none] (9) at (8, 0) {};
		\node [style=none] (10) at (6, 0) {};
		\node [style=none] (11) at (2, 2) {};
		\node [style=none] (12) at (4, 2) {};
		\node [style=none] (13) at (6, 2) {};
		\node [style=none] (14) at (8, 2) {};
		\node [style=none] (15) at (0, 0) {$=$};
		\node [style=bn] (16) at (-3.5, -0.5) {};
		\node [style=none] (17) at (-3.5, -2) {};
		\node [style=none] (18) at (-5, 1) {};
		\node [style=none] (19) at (-2, 1) {};
		\node [style=none] (20) at (-2, 2.5) {$X\otimes Y$};
		\node [style=none] (21) at (-3.5, -2.5) {$X\otimes Y$};
		\node [style=none] (22) at (-5, 2) {};
		\node [style=none] (23) at (-2, 2) {};
		\node [style=none] (24) at (2, 2.5) {$X$};
		\node [style=none] (25) at (4, 2.5) {$Y$};
		\node [style=none] (26) at (6, 2.5) {$X$};
		\node [style=none] (27) at (8, 2.5) {$Y$};
		\node [style=none] (28) at (3, -2.5) {$X$};
		\node [style=none] (29) at (7, -2.5) {$Y$};
	\end{pgfonlayer}
	\begin{pgfonlayer}{edgelayer}
		\draw [style=none] (1.center) to (2);
		\draw [style=none, bend left=45] (2) to (5.center);
		\draw [style=none, bend right=45] (2) to (4.center);
		\draw [style=none] (6.center) to (7);
		\draw [style=none, bend left=45] (7) to (10.center);
		\draw [style=none, bend right=45] (7) to (9.center);
		\draw [style=none] (11.center) to (5.center);
		\draw [style=none] (8.center) to (14.center);
		\draw [style=none, in=-90, out=90] (3.center) to (13.center);
		\draw [style=none, in=90, out=-90] (12.center) to (10.center);
		\draw (17.center) to (16);
		\draw [bend right=315] (16) to (18.center);
		\draw [bend right=45] (16) to (19.center);
		\draw (23.center) to (19.center);
		\draw (18.center) to (22.center);
	\end{pgfonlayer}
\end{tikzpicture}
		\end{equation}
		holds for all $X, Y \in \cC$.
	\end{definition}
	
	\begin{example}
		As far as this paper is concerned, the most relevant example is that of $\BorelStoch$, which is the category of standard Borel spaces and measurable Markov kernels. 
		It extends the objects of $\FinStoch$ by including both countably infinite measurable spaces as well as those isomorphic to $[0,1]$ with its Borel $\sigma$-algebra.
		Morphisms coincide with those of $\FinStoch$ on finite sets, but in general they are given by Markov kernels.
		That is, a morphism $f \colon A \to X$ is specified by a family of probability measures $f ( \ph | a ) $ over $X$, one for each $a \in A$, such that $f(S | \ph ) \colon A \to [0,1]$ is a measurable map for every measurable subset $S \subseteq X$.
		Sequential composition of Markov kernels $f \colon A \to X$ and $h\colon X \to Z$ is given by the Chapman-Kolmogorov equation as usual, which in our notation reads
		\begin{equation}
			(h \circ f) ( T | a) = \int_X h( T | x ) \, f( dx | a)
		\end{equation}
		for every $a \in A$ and every measurable $T \subseteq Z$. 
		For more details, see \cite[Section 4]{fritz2019synthetic}.
	\end{example}
	
	\subsection{Conditionals}
	
		In the introduction, we mentioned that one of the ingredients in our synthetic proof of de Finetti's Theorem is the existence of conditional probability distributions.
		Let us make this more precise in the context of Markov categories.
		In $\FinStoch$, given a stochastic matrix $f \colon A \to X \otimes Y$ with entries $f ( x, y | a )$, there is also a stochastic matrix $f_{|X}( y | x, a )$ satisfying
		\begin{equation}\label{eq:finstoch_cond}
			f_{|X}( y | x, a ) = \frac{f ( x, y | a )}{\sum_{y'} f ( x, y' | a )}
		\end{equation}
		whenever the denominator $\sum_{y'} f ( x, y' | a )$ is non-zero, and taking arbitrary values otherwise.
		This gives the corresponding probability of $Y = y$ given that $A$ and $X$ attain values $a$ and $x$, respectively.
		Equation~\eqref{eq:finstoch_cond} can be viewed as a version of Bayes' Theorem.
		One can also characterize $f_{|X}$ implicitly by
		\begin{equation}\label{eq:finstoch_cond2}
			f ( x, y | a ) = f_{|X}( y | x, a ) \, \sum_{y'} f ( x, y' | a )
		\end{equation}
		which, unlike equation \eqref{eq:finstoch_cond}, can be expressed in string diagrams.
		\begin{definition}[conditionals {\cite[Definition 11.5]{fritz2019synthetic}}]
			\label{def:conditional}
			Let $f \colon A \to X \tensor Y$ be a morphism in a Markov category $\cC$.
			A morphism $f_{|X} \colon X \tensor A \to Y$ in $\cC$ is called a \emph{conditional} of $f$ with respect to $X$ if the equation
			\begin{equation}\label{eq:conditional}
				\begin{tikzpicture}
	\begin{pgfonlayer}{nodelayer}
		\node [style=none] (22) at (0, 0) {$=$};
		\node [style=morphism] (54) at (-2.5, 0) {$\;\; f \;\; $};
		\node [style=none] (55) at (-2.5, -1) {};
		\node [style=none] (56) at (-2, 0) {};
		\node [style=none] (57) at (-2, 1) {};
		\node [style=none] (58) at (-3, 1) {};
		\node [style=none] (59) at (-3, 0) {};
		\node [style=none] (60) at (-3, 1.5) {$X$};
		\node [style=none] (61) at (-2, 1.5) {$Y$};
		\node [style=none] (62) at (-2.5, -1.5) {$A$};
		\node [style=morphism] (63) at (3, -0.5) {$\;f\;$};
		\node [style=none] (64) at (3.25, -0.5) {};
		\node [style=bn] (65) at (3.25, 0.5) {};
		\node [style=none] (66) at (2.75, -0.5) {};
		\node [style=bn] (67) at (2.75, 1) {};
		\node [style=bn] (68) at (3.75, -1.75) {};
		\node [style=none] (69) at (3.75, -2.25) {};
		\node [style=none] (70) at (4.5, -0.5) {};
		\node [style=none] (71) at (3.5, 2.25) {};
		\node [style=morphism] (72) at (3.75, 2.25) {$f_{|X}$};
		\node [style=none] (73) at (4, 2) {};
		\node [style=none] (74) at (3.75, 3.25) {};
		\node [style=none] (75) at (3.75, 3.75) {$Y$};
		\node [style=none] (76) at (2, 2.25) {};
		\node [style=none] (77) at (2, 3.25) {};
		\node [style=none] (78) at (2, 3.75) {$X$};
		\node [style=none] (79) at (3.75, -2.75) {$A$};
		\node [style=none] (80) at (4.5, 0) {};
		\node [style=none] (81) at (4, 2.25) {};
	\end{pgfonlayer}
	\begin{pgfonlayer}{edgelayer}
		\draw (55.center) to (54);
		\draw (59.center) to (58.center);
		\draw (56.center) to (57.center);
		\draw (64.center) to (65);
		\draw (69.center) to (68);
		\draw [in=-75, out=150] (68) to (63);
		\draw (66.center) to (67);
		\draw [in=-90, out=30] (68) to (70.center);
		\draw [in=-90, out=30] (67) to (71.center);
		\draw (72) to (74.center);
		\draw (76.center) to (77.center);
		\draw [in=-90, out=150] (67) to (76.center);
		\draw (70.center) to (80.center);
		\draw [in=-90, out=90, looseness=1.25] (80.center) to (73.center);
		\draw (73.center) to (81.center);
	\end{pgfonlayer}
\end{tikzpicture}
			\end{equation}
			holds.
			We say that $\cC$ \emph{has conditionals} if such a conditional exists for all morphisms $f \colon A \to X \otimes Y$ for any objects $A,X,Y$ in $\cC$.
		\end{definition}
		In $\BorelStoch$, this amounts to the existence of regular conditional probabilities for measurable Markov kernels~\cite[Example~2.4]{fritz2020representable}.
	
		As a special case of conditionals, we obtain a synthetic definition of a Bayesian inverse of $f \colon A \to X$ with respect to a prior $m \colon I \to A$.
		\begin{definition}[Bayesian inverse \cite{chojacobs2019strings}]
			\label{def:Bayesian_inverse}
			Given two morphisms $m \colon I \to A$ and ${f \colon A \to X}$, a \emph{Bayesian inverse} of $f$ with respect to $m$ is a conditional of
			\begin{equation}\label{eq:Bayesian_inverse}
				\begin{tikzpicture}
	\begin{pgfonlayer}{nodelayer}
		\node [style=bn] (12) at (0, -0.25) {};
		\node [style=none] (13) at (-1, 1) {};
		\node [style=none] (14) at (-1, 1) {};
		\node [style=none] (15) at (-1, 2) {};
		\node [style=none] (16) at (1, 2) {};
		\node [style=state] (18) at (0, -1.25) {$m$};
		\node [style=morphism] (19) at (1, 1) {$f$};
		\node [style=none] (20) at (1, 2.5) {$X$};
		\node [style=none] (21) at (-1, 2.5) {$A$};
	\end{pgfonlayer}
	\begin{pgfonlayer}{edgelayer}
		\draw [bend left] (12) to (14.center);
		\draw (18) to (12);
		\draw (14.center) to (15.center);
		\draw (19) to (16.center);
		\draw [bend right] (12) to (19);
	\end{pgfonlayer}
\end{tikzpicture}
			\end{equation}
			with respect to $X$. Explicitly, it is a morphism $f^\dagger:X\to A$ satisfying 
			\begin{equation}
				\begin{tikzpicture}
	\begin{pgfonlayer}{nodelayer}
		\node [style=bn] (12) at (-3, -0.25) {};
		\node [style=none] (13) at (-4, 1) {};
		\node [style=none] (14) at (-4, 1) {};
		\node [style=none] (15) at (-4, 2) {};
		\node [style=none] (16) at (-2, 2) {};
		\node [style=state] (18) at (-3, -1.25) {$m$};
		\node [style=morphism] (19) at (-2, 1) {$f$};
		\node [style=none] (20) at (-2, 2.5) {$X$};
		\node [style=none] (21) at (-4, 2.5) {$A$};
		\node [style=none] (22) at (0, 0) {$=$};
		\node [style=bn] (23) at (3, -0.25) {};
		\node [style=state] (24) at (3, -2.75) {$m$};
		\node [style=none] (25) at (4, 0.75) {};
		\node [style=none] (26) at (2, 0.75) {};
		\node [style=none] (27) at (4, 2) {};
		\node [style=none] (28) at (2, 2) {};
		\node [style=morphism] (29) at (2, 1) {$f^\dag$};
		\node [style=morphism] (30) at (3, -1.25) {$f$};
		\node [style=none] (31) at (4, 2.5) {$X$};
		\node [style=none] (32) at (2, 2.5) {$A$};
	\end{pgfonlayer}
	\begin{pgfonlayer}{edgelayer}
		\draw [in=-90, out=165] (12) to (14.center);
		\draw (18) to (12);
		\draw (14.center) to (15.center);
		\draw (19) to (16.center);
		\draw [bend right] (12) to (19);
		\draw (27.center) to (25.center);
		\draw [in=15, out=-90] (25.center) to (23);
		\draw [in=-90, out=165] (23) to (26.center);
		\draw (29) to (28.center);
		\draw (24) to (30);
		\draw (30) to (23);
	\end{pgfonlayer}
\end{tikzpicture}
			\end{equation}
		\end{definition}
		For example in $\FinStoch$, such a Bayesian inverse $f^{\dagger} \colon X \to A$ satisfies
		\begin{equation}
			f^{\dagger} (a | x) \, \sum_{a'} f(x | a') \, m (a') = f (x | a) \, m(a)
		\end{equation}
	
		In general, one should keep in mind that even though we denote a Bayesian inverse of $f$ by $f^{\dagger}$, it does depend non-trivially on the prior $m$.
		Moreover, conditionals and Bayesian inverses are generally not unique when they exist.
		In $\FinStoch$, this is because $f_{|X}( y | x, a )$ is arbitrary whenever $\sum_y f ( x, y | a ) = 0$ in equation \eqref{eq:finstoch_cond2}.
		However, one can show that conditionals (and therefore also Bayesian inverses) \emph{are} unique up to almost sure equality \cite[Proposition~13.6]{fritz2019synthetic}.
		\begin{definition}[a.s.-equality \cite{chojacobs2019strings}]
			\label{def:ase}
			Given $m \colon \Theta \to A$, we say that ${f,g \colon A \to X}$ are \emph{$m$-almost surely equal}, denoted by $f \ase{m} g$, if we have
			\begin{equation}\label{eq:ase}
				\begin{tikzpicture}
	\begin{pgfonlayer}{nodelayer}
		\node [style=bn] (12) at (-3, -0.25) {};
		\node [style=none] (13) at (-4, 1) {};
		\node [style=none] (14) at (-4, 1) {};
		\node [style=none] (15) at (-4, 2) {};
		\node [style=none] (16) at (-2, 2) {};
		\node [style=morphism] (19) at (-2, 1) {$f$};
		\node [style=none] (20) at (-2, 2.5) {$X$};
		\node [style=none] (21) at (-4, 2.5) {$A$};
		\node [style=none] (22) at (0, 0) {$=$};
		\node [style=morphism] (23) at (-3, -1.25) {$m$};
		\node [style=none] (24) at (-3, -2.25) {};
		\node [style=none] (25) at (-3, -2.75) {$\Theta$};
		\node [style=bn] (26) at (3, -0.25) {};
		\node [style=none] (27) at (2, 1) {};
		\node [style=none] (28) at (2, 1) {};
		\node [style=none] (29) at (2, 2) {};
		\node [style=none] (30) at (4, 2) {};
		\node [style=morphism] (31) at (4, 1) {$g$};
		\node [style=none] (32) at (4, 2.5) {$X$};
		\node [style=none] (33) at (2, 2.5) {$A$};
		\node [style=morphism] (34) at (3, -1.25) {$m$};
		\node [style=none] (35) at (3, -2.25) {};
		\node [style=none] (36) at (3, -2.75) {$\Theta$};
	\end{pgfonlayer}
	\begin{pgfonlayer}{edgelayer}
		\draw [bend left] (12) to (14.center);
		\draw (14.center) to (15.center);
		\draw (19) to (16.center);
		\draw [bend right] (12) to (19);
		\draw (24.center) to (23);
		\draw (23) to (12);
		\draw [bend left] (26) to (28.center);
		\draw (28.center) to (29.center);
		\draw (31) to (30.center);
		\draw [bend right] (26) to (31);
		\draw (35.center) to (34);
		\draw (34) to (26);
	\end{pgfonlayer}
\end{tikzpicture}
			\end{equation}
		\end{definition}
		We can interpret equation \eqref{eq:ase} as saying that $f$ and $g$ can only differ for events that are deemed irrelevant by $m$.
		For example, in $\BorelStoch$, $f$ and $g$ are \as{$m$} equal if and only if they are equal with probability 1 for every value of $\Theta$, i.e.\ if and only if
		\begin{equation}\label{eq:ase_Borel}
			\int_{S} f ( T | a ) \, m(da | \theta)  =  \int_{S} g ( T | a ) \, m(da | \theta)
		\end{equation}
		holds for all $\theta \in \Theta$ and for all measurable subsets $S$ of $A$ and $T$ of $X$ respectively.
	
	\subsection{Representability}
	
		Second on our list of proof ingredients is the ability to express a space of distributions on an object in a Markov category $\cC$ as an object in $\cC$ itself.
		For example, given an object $X = \{0,1\}$ in $\BorelStoch$, we would like there be an object $PX$ isomorphic to $[0,1]$ whose elements are themselves probability distributions over $X$.
		Indeed, $\BorelStoch$ allows for such a construction \cite[Example 3.19]{fritz2020representable}.
		However, this is not the case for $\FinStoch$ of course since $PX$ cannot be a finite set.
		
		Before stating the general definition of $PX$ more formally, we need to address the question of how to refer to ``elements'' of an object in a Markov category.
		After all, the objects do not come equipped with any underlying set a priori.
		In $\FinStoch$, we can identify the finite set $X$ with those morphisms $I \to X$ that are delta distributions.
		That is, they are morphisms describing no randomness---the deterministic ones.
		More generally, deterministic morphisms in $\FinStoch$ are the $\{0,1\}$-valued stochastic matrices.
		Intuitively, a deterministic $f$ can be characterized by the fact that applying it to two independent copies of its input is guaranteed to result in the same pair of output values as applying $f$ directly to the input and copying its output.
		\begin{definition}[deterministic morphism {\cite[Definition 10.1]{fritz2019synthetic}}]
			\label{def:deterministic}
			Let ${f \colon A \to X}$ be a morphism in $\cC$.
			We say that $f$ is \emph{deterministic} if it satisfies:
			\begin{equation}\label{eq:deterministic}
				\begin{tikzpicture}
	\begin{pgfonlayer}{nodelayer}
		\node [style=none] (0) at (-3.5, -1.75) {};
		\node [style=bn] (1) at (-3.5, -0.75) {};
		\node [style=none] (2) at (-4.5, 0.5) {};
		\node [style=none] (3) at (-2.5, 0.5) {};
		\node [style=morphism] (5) at (-4.5, 0.75) {$f$};
		\node [style=morphism] (6) at (-2.5, 0.75) {$f$};
		\node [style=none] (7) at (3.25, -1.75) {};
		\node [style=bn] (8) at (3.25, 0.25) {};
		\node [style=none] (9) at (2.25, 1.5) {};
		\node [style=none] (11) at (4.25, 1.5) {};
		\node [style=none] (12) at (-4.5, 1.75) {};
		\node [style=none] (13) at (-2.5, 1.75) {};
		\node [style=none] (14) at (2.25, 1.75) {};
		\node [style=none] (15) at (4.25, 1.75) {};
		\node [style=none] (16) at (0, 0) {$=$};
		\node [style=morphism] (17) at (3.25, -0.75) {$f$};
		\node [style=none] (18) at (-4.5, 2.25) {$X$};
		\node [style=none] (19) at (-2.5, 2.25) {$X$};
		\node [style=none] (20) at (2.25, 2.25) {$X$};
		\node [style=none] (21) at (4.25, 2.25) {$X$};
		\node [style=none] (22) at (-3.5, -2.25) {$A$};
		\node [style=none] (23) at (3.25, -2.25) {$A$};
	\end{pgfonlayer}
	\begin{pgfonlayer}{edgelayer}
		\draw [style=none] (0.center) to (1);
		\draw [style=none, in=270, out=165] (1) to (2.center);
		\draw [style=none, in=270, out=15] (1) to (3.center);
		\draw [style=none] (7.center) to (8);
		\draw [style=none, in=270, out=165] (8) to (9.center);
		\draw (12.center) to (5);
		\draw (13.center) to (6);
		\draw (14.center) to (9.center);
		\draw (15.center) to (11.center);
		\draw (2.center) to (5);
		\draw (3.center) to (6);
		\draw [in=-90, out=15] (8) to (11.center);
	\end{pgfonlayer}
\end{tikzpicture}
			\end{equation}
			The subcategory of $\cC$ that consists of its deterministic morphisms is denoted by $\cC_{\det}$.
		\end{definition}
		Indeed, condition \eqref{eq:deterministic} fails for every stochastic matrix that is not $\{0,1\}$-valued.
		Besides deterministic morphisms, we will also make use of the concept of $m$-almost surely deterministic ones.
		\begin{definition}[\as{}-deterministic morphism {\cite[Definition 13.11]{fritz2019synthetic}}]
			\label{def:as_det}
			A morphism $f$ in $\cC$ is \emph{$m$-almost surely deterministic} if it satisfies:
			\begin{equation}\label{eq:as_det}
				\begin{tikzpicture}
	\begin{pgfonlayer}{nodelayer}
		\node [style=bn] (1) at (-4.75, 0.5) {};
		\node [style=none] (2) at (-5.75, 1.75) {};
		\node [style=none] (3) at (-3.75, 1.75) {};
		\node [style=morphism] (5) at (-5.75, 2) {$f$};
		\node [style=morphism] (6) at (-3.75, 2) {$f$};
		\node [style=bn] (8) at (3.25, 1.5) {};
		\node [style=none] (9) at (2.25, 2.75) {};
		\node [style=none] (11) at (4.25, 2.75) {};
		\node [style=none] (12) at (-5.75, 3) {};
		\node [style=none] (13) at (-3.75, 3) {};
		\node [style=none] (14) at (2.25, 3) {};
		\node [style=none] (15) at (4.25, 3) {};
		\node [style=none] (16) at (0, 0) {$=$};
		\node [style=morphism] (17) at (3.25, 0.5) {$f$};
		\node [style=none] (18) at (-5.75, 3.5) {$X$};
		\node [style=none] (19) at (-3.75, 3.5) {$X$};
		\node [style=none] (20) at (2.25, 3.5) {$X$};
		\node [style=none] (21) at (4.25, 3.5) {$X$};
		\node [style=bn] (24) at (-3.5, -1) {};
		\node [style=none] (25) at (-2.25, 0.5) {};
		\node [style=none] (26) at (-2.25, 3) {};
		\node [style=none] (27) at (-2.25, 3.5) {$A$};
		\node [style=morphism] (28) at (-3.5, -2) {$m$};
		\node [style=none] (29) at (-3.5, -3) {};
		\node [style=none] (30) at (-3.5, -3.5) {$\Theta$};
		\node [style=bn] (31) at (4.5, -1) {};
		\node [style=none] (32) at (5.75, 0.5) {};
		\node [style=none] (33) at (5.75, 3) {};
		\node [style=none] (34) at (5.75, 3.5) {$A$};
		\node [style=morphism] (35) at (4.5, -2) {$m$};
		\node [style=none] (36) at (4.5, -3) {};
		\node [style=none] (37) at (4.5, -3.5) {$\Theta$};
		\node [style=none] (38) at (3.25, 0.25) {};
	\end{pgfonlayer}
	\begin{pgfonlayer}{edgelayer}
		\draw [style=none, in=270, out=165] (1) to (2.center);
		\draw [style=none, in=270, out=15] (1) to (3.center);
		\draw [style=none, in=270, out=165] (8) to (9.center);
		\draw (12.center) to (5);
		\draw (13.center) to (6);
		\draw (14.center) to (9.center);
		\draw (15.center) to (11.center);
		\draw (2.center) to (5);
		\draw (3.center) to (6);
		\draw [in=-90, out=15] (8) to (11.center);
		\draw [in=-90, out=165] (24) to (1);
		\draw [in=-90, out=15] (24) to (25.center);
		\draw (25.center) to (26.center);
		\draw (29.center) to (28);
		\draw (28) to (24);
		\draw [in=-90, out=15] (31) to (32.center);
		\draw (32.center) to (33.center);
		\draw (36.center) to (35);
		\draw (35) to (31);
		\draw (17) to (8);
		\draw [in=-90, out=165] (31) to (38.center);
	\end{pgfonlayer}
\end{tikzpicture}
			\end{equation}
		\end{definition}	
		
		Given Definition~\ref{def:deterministic}, we thus identify ``elements'' of the hypothetical space of distributions $PX$ with deterministic morphisms $I \to PX$.
		In order for these to faithfully represent probability distributions on $X$ (and nothing else), in the general categorical setting we therefore require there to be a bijection between $\cC(I,X)$ and $\cC_\det(I,PX)$.
		Extending this requirement to morphisms with arbitrary domain $A$, which is expected to hold by the same reasoning, leads to the definition of distribution objects.
		\begin{definition}[representable Markov category {\cite[Definition 3.7]{fritz2020representable}}]
		\label{def:representability}
			Given an object $X$ in a Markov category $\cat{C}$, a \emph{distribution object} for $X$ is an object $PX$ together with natural bijections\footnote{The bijections being \emph{natural} refers to the property that the functions instantiating them for different choices of $A$ make up the components of a \href{https://ncatlab.org/nlab/show/natural+isomorphism}{natural isomorphism}.}%
			\begin{equation}\label{eq:representability}
				\cat{C}(A,X) \cong \cat{C}_\det(A,PX) 
			\end{equation}
			between morphisms into $X$ and deterministic morphisms into $PX$.
			Given any morphism $f \colon A\to X$, we denote by $f^\sharp \colon A\to PX$ its counterpart under this bijective correspondence. 	 
			We say that $\cat{C}$ is \emph{representable} if every object of $\cat{C}$ has a distribution object.
		\end{definition}
		
		It turns out that a slightly stronger version of representability is needed in the synthetic proof of de Finetti's Theorem.
		In particular, we require that the identification from \eqref{eq:representability} is compatible with \as{}-equality in the following sense.
		\begin{definition}[\as{}-compatible representability {\cite[Definition 3.18]{fritz2020representable}}]
			\label{def:as_comp_rep}
			A representable Markov category $\cat{C}$ is called \emph{\as{}-compatibly representable} if we have
			\begin{equation}\label{eq:as_comp_rep}
				\begin{tikzpicture}
	\begin{pgfonlayer}{nodelayer}
		\node [style=morphism] (0) at (-9.5, 0) {$f$};
		\node [style=none] (1) at (-9.5, 1.25) {};
		\node [style=none] (2) at (-9.5, -1.25) {};
		\node [style=none] (3) at (-9.5, 1.75) {$X$};
		\node [style=none] (4) at (-9.5, -1.75) {$A$};
		\node [style=morphism] (6) at (-4.5, 0) {$g$};
		\node [style=none] (7) at (-4.5, 1.25) {};
		\node [style=none] (8) at (-4.5, -1.25) {};
		\node [style=none] (9) at (-4.5, 1.75) {$X$};
		\node [style=none] (10) at (-4.5, -1.75) {$A$};
		\node [style=none] (11) at (0, 0) {$\iff$};
		\node [style=morphism] (12) at (4.5, 0) {$f^\sharp$};
		\node [style=none] (13) at (4.5, 1.25) {};
		\node [style=none] (14) at (4.5, -1.25) {};
		\node [style=none] (15) at (4.5, 1.75) {$PX$};
		\node [style=none] (16) at (4.5, -1.75) {$A$};
		\node [style=none] (17) at (7.25, 0) {$\ase{m}$};
		\node [style=morphism] (18) at (9.5, 0) {$g^\sharp$};
		\node [style=none] (19) at (9.5, 1.25) {};
		\node [style=none] (20) at (9.5, -1.25) {};
		\node [style=none] (21) at (9.5, 1.75) {$PX$};
		\node [style=none] (22) at (9.5, -1.75) {$A$};
		\node [style=none] (23) at (-6.75, 0) {$\ase{m}$};
	\end{pgfonlayer}
	\begin{pgfonlayer}{edgelayer}
		\draw (2.center) to (0);
		\draw (0) to (1.center);
		\draw (8.center) to (6);
		\draw (6) to (7.center);
		\draw (12) to (13.center);
		\draw (14.center) to (12);
		\draw (20.center) to (18);
		\draw (18) to (19.center);
	\end{pgfonlayer}
\end{tikzpicture}
			\end{equation}
			for all $m \colon \Theta \to A$ and all $f,g$ as indicated.
		\end{definition}
		For example, $\BorelStoch$ is \as{}-compatibly representable~\cite[Example~3.19]{fritz2020representable}.
	
	    If we set $A = PX$ in bijection \eqref{eq:representability}, we get a correspondence between deterministic morphisms $PX\to PX$ and generic morphisms $PX\to X$. 
	    The identity on $PX$ corresponds to a map $PX \to X$ which we denote by $\samp$, and which we can think of as taking a probability distribution $p$ on $X$ and returning a random element of $X$ distributed according to $p$ (hence, ``sampling'' from $p$). 
	    In $\cat{BorelStoch}$, for instance, for every $p\in PX$ and every measurable $S \subseteq X$ we have  
	    \begin{equation}
			\samp(S|p) = p(S) .
		\end{equation}
		By naturality of the bijection \eqref{eq:representability}, for each morphism $f \colon A\to X$ we have that $f = \samp \circ f^\sharp$ holds. 
		For more details on this, see \cite[Section 3]{fritz2020representable}.
	
	\subsection{Infinite Products}
	
		The third and last ingredient that we need in the proof of our synthetic de Finetti's Theorem is the ability to construct countable products of objects.
		Once again we would not expect $\FinStoch$ to allow those, since its objects are merely finite sets, but $\BorelStoch$ does \cite[Example 3.6]{fritzrischel2019zeroone}.
		The relevant definition is that of \emph{Kolmogorov products} introduced in \cite{fritzrischel2019zeroone} in the context of $0/1$-laws for Markov categories.
	
		Given a hypothetical object $X^{\N}$ describing the product of countably many copies of $X$, every probability measure $f \colon I \to X^{\N}$ should give rise to a corresponding measure $f_{F} \colon I \to X^F$ for each finite subset $F \subseteq \N$.
		Intuitively, $f_F$ is given by marginalizing over those copies of $X$ indexed by the set complement of $F$.
		Therefore, we might expect there to be a corresponding deterministic morphism $\pi_F \colon X^{\N} \to X^F$, which gives $f_F$ when composed with $f$.
		Conversely, the Kolmogorov extension theorem suggests that one can reconstruct $f$ uniquely from its family of finite marginals $(f_F)$, provided that these are suitably compatible.
		That is to say, if $F' \subseteq F$ is a further subset, then one can obtain $f_{F'}$ from $f_F$ by marginalizing over the extraneous copies of $X$. 
		Under this condition, there should be a unique $f$ such that $f_F = \pi_F \circ f$ for all $F$.
		
		The idea underlying Kolmogorov products is to turn these properties that one would expect from a countable product into a definition of a meaningful product of infinitely many objects in any Markov category.
		However, besides merely morphisms of type $I \to X^\N$, we also require similar properties of those with a non-trivial domain $A$ and an additional codomain $Y$ that could be correlated with $X^\N$.
		\begin{definition}[Kolmogorov powers {\cite[Definition 4.1]{fritzrischel2019zeroone}}]
			\label{def:inf_product}
			Given an object $X$ in a Markov category $\cat{C}$, consider an object $X^\N$ of $\cC$ such that there is a natural bijection between 
			\begin{itemize}
				\item morphisms $f$ in $\cC( A, X^\N \otimes Y)$ and
				\item families $(f_F \colon A \to X^F \otimes Y)$ of morphisms indexed by finite subsets $F$ of $\N$ that are compatible in the sense that whenever $F'$ is a subset of $F$, we have
				\begin{equation}
					\begin{tikzpicture}
	\begin{pgfonlayer}{nodelayer}
		\node [style=none] (16) at (0, 0) {$=$};
		\node [style=morphism] (17) at (-3.5, 0) {$\quad f_{F'} \quad$};
		\node [style=morphism] (18) at (3.5, -0.75) {$\quad f_{F} \quad$};
		\node [style=none] (19) at (-4.5, 0) {};
		\node [style=none] (20) at (-2.5, 0) {};
		\node [style=none] (21) at (-4.5, 2) {};
		\node [style=none] (22) at (-2.5, 2) {};
		\node [style=none] (23) at (-4.5, 2.5) {$X^{F'}$};
		\node [style=none] (24) at (-2.5, 2.5) {$Y$};
		\node [style=none] (25) at (-3.5, -2) {};
		\node [style=none] (26) at (-3.5, -2.5) {$A$};
		\node [style=none] (27) at (2.5, -0.75) {};
		\node [style=none] (28) at (4.5, -0.75) {};
		\node [style=morphism] (29) at (2.5, 1) {$\pi_{F,F'}$};
		\node [style=none] (30) at (2.5, 2) {};
		\node [style=none] (31) at (2.5, 2.5) {$X^{F'}$};
		\node [style=none] (32) at (4.5, 2) {};
		\node [style=none] (33) at (4.5, 2.5) {$Y$};
		\node [style=none] (34) at (3.5, -2) {};
		\node [style=none] (35) at (3.5, -2.5) {$A$};
	\end{pgfonlayer}
	\begin{pgfonlayer}{edgelayer}
		\draw (25.center) to (17);
		\draw (19.center) to (21.center);
		\draw (20.center) to (22.center);
		\draw (34.center) to (18);
		\draw (27.center) to (29);
		\draw (29) to (30.center);
		\draw (28.center) to (32.center);
	\end{pgfonlayer}
\end{tikzpicture}
				\end{equation}
				where $\pi_{F,F'} \colon X^F \to X^{F'}$ acts as the identity on $X^{F'}$ and applies $\discard_X$ to each of the remaining factors,
			\end{itemize}
			for all objects $A$ and $Y$ of $\cC$.
			If every morphism in the family $(f_F)$ corresponding to the case of $Y = I$, $A = X^\N$, and $f = \id_{X^\N}$ is deterministic, then $X^\N$ is termed a \emph{countable Kolmogorov power} of $X$.
		\end{definition}
		For example, Kolmogorov's extension theorem (in the countable case) states that countable Kolmogorov powers exist in $\BorelStoch$~\cite[Example~3.6]{fritzrischel2019zeroone}.

		Note that the family of morphisms that corresponds to $\id \colon X^\N \to X^\N$ under the prescribed bijection is given by the ``infinite marginalization maps'' $\pi_F \colon X^\N \to X^F$ mentioned before, which play the role of product projections.

		If the Kolmogorov powers $X^\N$ and $Y^\N$ exist for two objects $X$ and $Y$, then for any morphism $f \colon X\to Y$ we get a canonical morphism $X^\N \to Y^\N$ as follows.
		First, for each finite subset $F \subset \N$ we form the map $f^F \colon X^F\to Y^F$ as usual. 
		The family of compositions
		\begin{equation*}
			\begin{tikzcd}
			 	X^\N \ar{r}{\pi_F} & X^F \ar{r}{f^F} & Y^F
			\end{tikzcd}
		\end{equation*}
		for each $F$, by the bijection of Definition~\ref{def:inf_product}, corresponds to a unique morphism $X^\N\to Y^\N$.
		We denote this morphism by $f^\N$.
		
		Although the original definition in \cite{fritzrischel2019zeroone} defines Kolmogorov products of any arbitrarily large family of objects, in the present context we limit ourselves to countable powers of one and the same object, since this is all we need for de Finetti's Theorem.
		
			In the following sections, we use double wires in order to distinguish an object that is a countable Kolmogorov power (such as $X^\N$) from a generic object of a Markov category (such as $X$) depicted by a single wire.
			See equation \eqref{eq:exchangeability} for example.
			At times, we also use three dots to indicate a countable Kolmogorov power, such as in equation \eqref{eq:de_finetti_claim}.
			The latter notation is motivated by the fact that a countable Kolmogorov power $f^\N$ of a morphism $f$ can be informally viewed as a countable number of parallel morphisms, each given by $f$.

\section{De Finetti's Theorem for Markov Categories}
\label{sec:synthetic}
	
	Suppose that a Markov category $\cC$ has countable Kolmogorov powers.
	We can then express what it means for a morphism with a countable power $X^\N$ as codomain to be exchangeable.
	Note that for any injective function $i \colon \N \to \N$ (and in particular for any permutation) we can define a morphism $X^i \colon X^\N \to X^\N$ that maps the $n$-th component in $X^\N$ to 
	\begin{itemize}
		\item the $i^{-1}(n)$-th component of $X^\N$ whenever $n$ is in the image of $i$, and
		\item discards it otherwise.
	\end{itemize}
	The construction of $X^i$ is by the definition of $X^\N$ as the Kolmogorov power of $X$.
	That is, $X^i$ corresponds to the family of morphisms $(X^i_F \colon X^\N \to X^F)$ as in Definition~\ref{def:inf_product}, where $X^i_F$ is the composite of the product projection ${\pi_{i(F)} \colon X^\N \to X^{i(F)}}$ with the canonical isomorphism $X^{i(F)} \cong X^F$ resulting from identifying the factors as prescribed by $i$.

	\begin{definition}[exchangeability]
		\label{def:exchangeability}
		A morphism $p \colon A \to X^\N$ is said to be \emph{exchangeable} if it is invariant under finite permutations of the components of the Kolmogorov power $X^\N$.
		That is, we demand that for every finite permutation $\sigma \colon \N \to \N$, we have
		\begin{equation}\label{eq:exchangeability}
			\begin{tikzpicture}
	\begin{pgfonlayer}{nodelayer}
		\node [style=none] (11) at (0, 0) {$=$};
		\node [style=morphism] (12) at (-2.5, 0) {$p$};
		\node [style=none] (13) at (-2.5, 1.75) {};
		\node [style=none] (15) at (-2.5, -1.75) {};
		\node [style=morphism] (16) at (2.5, -0.75) {$p$};
		\node [style=morphism] (17) at (2.5, 0.75) {$X^{\sigma}$};
		\node [style=none] (18) at (2.5, 1.75) {};
		\node [style=none] (19) at (2.5, -1.75) {};
	\end{pgfonlayer}
	\begin{pgfonlayer}{edgelayer}
		\draw (15.center) to (12);
		\draw [style=double wire] (12) to (13.center);
		\draw (19.center) to (16);
		\draw [style=double wire] (16) to (17);
		\draw [style=double wire] (17) to (18.center);
	\end{pgfonlayer}
\end{tikzpicture}
		\end{equation}
	\end{definition}
	
	\begin{assumption}
		\label{three_assumptions}
		Unless stated otherwise, throughout the rest of this paper we assume that $\cC$ is a Markov category that:
		\begin{enumerate}
			\item has conditionals, 
			\item is \as{}-compatibly representable, and
			\item has countable Kolmogorov powers. 
		\end{enumerate}
	\end{assumption}
	
	Note that Assumption \ref{three_assumptions} implies, in particular, that for every object $X$, the space $(PX)^\N$ exists, and is equipped with a copy map $(PX)^\N\to(PX)^\N\otimes (PX)^\N$.
	
	\begin{example}
		$\BorelStoch$ satisfies these assumptions.
		The relevant arguments can be found in \cite[Theorem~1.25]{kallenberg2017random}, \cite[Example 3.19]{fritz2020representable}, and \cite[Example 3.6]{fritzrischel2019zeroone} respectively.
	\end{example}
	
	\begin{theorem}[synthetic de Finetti's Theorem]\label{syntheticthm}
		Let $\cC$ be a Markov category satisfying Assumption~\ref{three_assumptions}.
		Then a morphism $p \colon A \to X^\N$ in $\cC$ is exchangeable if and only if there is a morphism $\mu \colon A \to PX$ such that we have
		\begin{equation}\label{eq:de_finetti_claim}
			\begin{tikzpicture}
	\begin{pgfonlayer}{nodelayer}
		\node [style=bn] (0) at (4.5, -0.5) {};
		\node [style=none] (2) at (2.5, 1.5) {};
		\node [style=none] (3) at (6.5, 1.5) {};
		\node [style=morphism] (4) at (6.5, 1.5) {$\samp$};
		\node [style=morphism] (5) at (2.5, 1.5) {$\samp$};
		\node [style=none] (6) at (4.5, 1.5) {${\cdots}$};
		\node [style=none] (7) at (2.5, 2.5) {};
		\node [style=none] (8) at (6.5, 2.5) {};
		\node [style=morphism] (9) at (4.5, -1.5) {$\,\mu\,$};
		\node [style=none] (10) at (4.5, -2.5) {};
		\node [style=none] (11) at (-3.25, -2.25) {};
		\node [style=none] (12) at (-2.25, 2.25) {};
		\node [style=medium box] (13) at (-3.25, 0) {$\;\;p\;\;$};
		\node [style=none] (14) at (-4.25, 2.25) {};
		\node [style=none] (15) at (0, 0) {$=$};
		\node [style=none] (16) at (-3.75, 0.5) {};
		\node [style=none] (17) at (-2.75, 0.5) {};
		\node [style=none] (18) at (-3.25, 1.75) {${\cdots}$};
	\end{pgfonlayer}
	\begin{pgfonlayer}{edgelayer}
		\draw [in=165, out=-90] (2.center) to (0);
		\draw [in=-90, out=15] (0) to (3.center);
		\draw (8.center) to (4);
		\draw (7.center) to (5);
		\draw (0) to (9);
		\draw (10.center) to (9);
		\draw (11.center) to (13);
		\draw [in=90, out=-90] (12.center) to (17.center);
		\draw [in=-90, out=90] (16.center) to (14.center);
	\end{pgfonlayer}
\end{tikzpicture}
		\end{equation}
	\end{theorem}
	
	In particular, this result says that the outputs of $p$ are conditionally independent and identically distributed.
	As we will see in the proof, the conditioning here is with respect to a synthetic version of the tail $\sigma$-algebra, thereby reproducing this classical aspect of de Finetti's Theorem as well.
	In fact, the proof uses an abstract version of the well-known result of measure-theoretic probability that the exchangeable \mbox{$\sigma$-algebra} and the tail $\sigma$-algebra coincide up to sets of measure zero~\cite[Corollary~12.18]{klenke2014probability}.
	In the following synthetic version of this result, the role of these $\sigma$-algebras is played by all possible ways of ``probing'' joint distributions of $X^\N$ by composing with a deterministic (or almost surely deterministic) \mbox{morphism $f$}.
	\begin{proposition}\label{prop:tail_algebra}
		Let $\cC$ be a Markov category with conditionals and countable Kolmogorov powers.
		Let $p \colon A \to X^\N$ be exchangeable, and suppose that ${f \colon X^\N \to Y}$ is $\as{p}$ deterministic. 
		Then the following are equivalent:
		\begin{enumerate}
			\item Finite permutation invariance: For every finite permutation $\sigma$, we have
				\begin{equation}\label{eq:perm_invariance}
					f \ase{p} f \circ X^\sigma.
				\end{equation}
				
			\item Shift invariance: For the successor function $s \colon \N \to \N$ mapping $n$ to $n+1$, we have
				\begin{equation}\label{eq:shift_invariance}
					f \ase{p} f \circ X^s.
				\end{equation}
		\end{enumerate}
	\end{proposition}
	The successor function can be equivalently defined as the inclusion of the second component in the coproduct decomposition $\N \cong 1 + \N$ and we depict its action on the Kolmogorov power $X^\N$ by either of the two following string diagrams:
	\begin{equation*}\label{eq:shifting}
		\begin{tikzpicture}
	\begin{pgfonlayer}{nodelayer}
		\node [style=none] (3) at (0, 0) {$=$};
		\node [style=bn] (6) at (2, 0.75) {};
		\node [style=none] (7) at (2.5, -1.25) {};
		\node [style=none] (9) at (2, 0.5) {};
		\node [style=none] (10) at (3, 0.5) {};
		\node [style=none] (11) at (2.5, -0.5) {};
		\node [style=none] (12) at (2.465, -0.5) {};
		\node [style=none] (13) at (3, 1.25) {};
		\node [style=morphism] (14) at (-2.5, 0) {$X^s$};
		\node [style=none] (16) at (-2.5, -1.25) {};
		\node [style=none] (17) at (-2.5, 1.25) {};
	\end{pgfonlayer}
	\begin{pgfonlayer}{edgelayer}
		\draw [style=double wire] (7.center) to (11.center);
		\draw [style=double wire, in=-90, out=90] (11.center) to (10.center);
		\draw [in=-90, out=90] (12.center) to (9.center);
		\draw (9.center) to (6);
		\draw [style=double wire] (10.center) to (13.center);
		\draw [style=double wire] (16.center) to (14);
		\draw [style=double wire] (14) to (17.center);
	\end{pgfonlayer}
\end{tikzpicture}
	\end{equation*}
	This morphism effectively discards the first component of $X^\N$. 
	
\begin{remark}
	While Theorem~\ref{syntheticthm} captures the key part of de Finetti's Theorem characterizing exchangeable morphisms, we have not been able to prove that our assumptions imply the uniqueness of the ``de Finetti measure'' $\mu$ in equation~\eqref{eq:de_finetti_claim}.
	This property \emph{is} often proven as part of classical versions of de Finetti's Theorem including the one of Hewitt and Savage for compact Hausdorff spaces~\cite{hewitt1955symmetric}.
\end{remark}
	
	Here is how we can use Theorem~\ref{syntheticthm} to obtain the classical de Finetti's Theorem in the form of Theorem~\ref{classicalthm}. 
	First of all, we take as $\cat{C}$ the category $\cat{BorelStoch}$, which satisfies the relevant assumptions. 
	Moreover, since the usual statement is given for the case of probability measures as opposed to Markov kernels with nontrivial domain, it suffices to consider the case of $A = I$. 
	
	Then the left-hand side of equation \eqref{eq:de_finetti_claim} instantiated in our case is an exchangeable probability measure on $X^\N$ in the sense of Section~\ref{sec:classical}.
	Theorem~\ref{syntheticthm} says that there exists a Markov kernel $\mu \colon I \to PX$ (equivalently, a probability measure on $PX$) such that $p$ can be written as in the right-hand side of equation~\eqref{eq:de_finetti_claim}. 
	Evaluating this equation on a cylinder set corresponding to a sequence of measurable subsets $S_1, \dots, S_n$ of $X$ gives
	\begin{align*}
		p \mleft( S_1\times \dots \times S_n \mright) 
			&= \int_{PX} \mleft( \prod_{i=1}^n \int_{PX} \samp (S_i | q_i) \mright) \cop(dq_1 \otimes \dots \otimes dq_n | q) \,  \mu (dq)  \\
			&= \int_{PX} \prod_{i=1}^n \samp (S_i | q) \, \mu (dq)  \\
			&= \int_{PX} q(S_1) \times \cdots \times q(S_n) \, \mu (dq),
	\end{align*}
	where we used the fact that $ \samp (S | q) = q(S)$. 
	Thus, Theorem~\ref{classicalthm} follows from Theorem~\ref{syntheticthm}.

	Similarly, here is how Proposition~\ref{prop:tail_algebra} recovers the classical coincidence of exchangeable $\sigma$-algebra and tail $\sigma$-algebra up to null sets.
	In $\BorelStoch$, a deterministic morphism $f \colon X^\N \to \{0,1\}$ corresponds to an event in $X^\N$, and such $f$ satisfies the shift invariance condition if and only if the corresponding event is, up to null sets, in the tail $\sigma$-algebra. 
	Likewise, such $f$ satisfies finite permutation invariance if and only if the event is, up to null sets, in the exchangeable $\sigma$-algebra.

	\begin{remark}
		Note that by making the same argument with $A$ an arbitrary standard Borel space, one obtains a parametric version of the classical Theorem~\ref{classicalthm}.
		It says that if the exchangeable probability measure $p$ depends measurably on a parameter, then $\mu$ can also be chosen to depend measurably on the same parameter.
		This is the de Finetti theorem for exchangeable Markov kernels mentioned in the introduction.
	\end{remark}

	\begin{remark}
		It is conceivable that versions of de Finetti's Theorem applying to larger classes of spaces, such as the result of Hewitt and Savage for exchangeable Radon probability measures on compact Hausdorff spaces, can also be obtained from Theorem~\ref{syntheticthm} by instantiating it in a suitable Markov category. 
		The recent results of Forr\'e~\cite{forre2021conditional} may be relevant here to establish the existence of conditionals.
	\end{remark}

\section{Diagrammatic Proof of de Finetti's Theorem}
\label{sec:proof}
	
	\newcommand{\flowchartbox}[2]{\fbox{\parbox{#1}{\centering\textrm{#2}}}}
	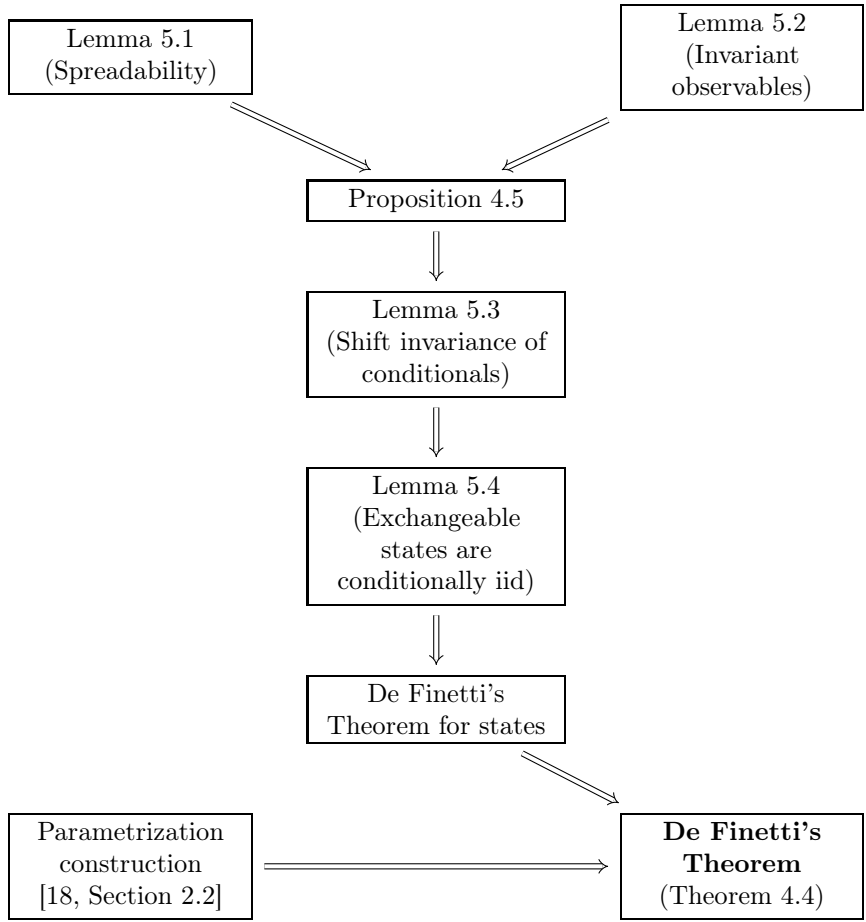
\begin{figure}
		\[\begin{tikzcd}[ampersand replacement=\&, row sep=1.6pc, column sep=12pt, every arrow/.append style={Rightarrow}]
			\flowchartbox{3cm}{Lemma~\ref{lem:spreadability} (Spreadability)} \&\& \flowchartbox{3cm}{Lemma~\ref{lem:invariant_observable} (Invariant observables)} \\
			\& \flowchartbox{3.2cm}{Proposition~\ref{prop:tail_algebra}} \\
			\& \flowchartbox{3.2cm}{Lemma~\ref{lem:tail_is_tail} \\ (Shift invariance of conditionals)} \\
			\& \flowchartbox{3.2cm}{Lemma~\ref{lem:conditionally_iid} (Exchangeable states are conditionally iid)} \\
			\& \flowchartbox{3.2cm}{De Finetti's Theorem for states} \\
			\flowchartbox{3cm}{Parametrization construction \\ \cite[Section 2.2]{fritz2020representable}} \&\& \flowchartbox{3cm}{\textbf{De Finetti's Theorem} \\ (Theorem~\ref{syntheticthm})}
			\arrow[from=1-1, to=2-2]
			\arrow[from=1-3, to=2-2]
			\arrow[from=2-2, to=3-2]
			\arrow[from=3-2, to=4-2]
			\arrow[from=4-2, to=5-2]
			\arrow[from=5-2, to=6-3]
			\arrow[from=6-1, to=6-3]
		\end{tikzcd}\]
		\caption{Overall structure of the proof of Theorem~\ref{syntheticthm}.}
		\label{fig:proof_flowchart}
	\end{figure}
	
	The overall structure of the proof is presented in Figure~\ref{fig:proof_flowchart}.
	We first present a number of lemmas that may be useful in other contexts too.
	We assume throughout that we are in a Markov category $\cC$ satisfying Assumption~\ref{three_assumptions}.
	
	\begin{lemma}[spreadability]
		\label{lem:spreadability}
		Let $p \colon A \to Y \otimes X^\N$ be exchangeable in the second factor. 
		Then for every injective function $i \colon \N \to \N$, we also have
		\begin{equation}\label{eq:spreadability}
			\begin{tikzpicture}
	\begin{pgfonlayer}{nodelayer}
		\node [style=medium box] (0) at (-3.25, -0.75) {$\quad p \quad$};
		\node [style=morphism] (1) at (-2.5, 1) {$X^i$};
		\node [style=none] (2) at (-2.5, -0.5) {};
		\node [style=none] (3) at (-4, -0.5) {};
		\node [style=none] (4) at (-2.5, 2) {};
		\node [style=none] (5) at (-3.25, -2) {};
		\node [style=none] (6) at (-4, 2) {};
		\node [style=medium box] (7) at (2.5, -0.75) {$\quad p \quad$};
		\node [style=none] (8) at (3.25, -0.5) {};
		\node [style=none] (9) at (1.75, -0.5) {};
		\node [style=none] (10) at (2.5, -2) {};
		\node [style=none] (11) at (3.25, 2) {};
		\node [style=none] (12) at (1.75, 2) {};
		\node [style=none] (13) at (0, 0) {$=$};
	\end{pgfonlayer}
	\begin{pgfonlayer}{edgelayer}
		\draw [style=double wire] (2.center) to (1);
		\draw [style=double wire] (4.center) to (1);
		\draw (5.center) to (0);
		\draw (3.center) to (6.center);
		\draw (10.center) to (7);
		\draw (12.center) to (9.center);
		\draw [style=double wire] (11.center) to (8.center);
	\end{pgfonlayer}
\end{tikzpicture}
		\end{equation}
	\end{lemma}
	
	In the concrete setting of $\cat{BorelStoch}$, this statement reads as follows in the case $A = Y = I$.
	Consider a probability measure $p\in P(X^\N)$ specifying an exchangeable joint distribution of infinitely many random variables. 
	If we marginalize over the first component, then the resulting distribution is again $p$. 
	In other words, if $p$ is exchangeable, then it is also shift-invariant.
	More generally, even if we marginalize over any number of components whose indices are specified by the complement of the image of some injective function $i \colon \N \to \N$, the resulting distribution is still $p$.
	The same logic can be applied in the presence of a non-trivial covariate $Y$ and parameter $A$, i.e.\ when $p$ is a Markov kernel $A \to Y \otimes X^\N$. 
	
	The relevant argument was essentially already given in the proof of~\cite[Section~5]{fritzrischel2019zeroone}, but we reproduce the argument here in the present more general context. 
	Note that this works the same way for any other infinite set in place of $\N$ (if the relevant Kolmogorov power exists).
	
	\begin{proof}
		Since $X^\N$ is the Kolmogorov power of $X$, it suffices to prove equation \eqref{eq:spreadability} after composition with each of the infinite marginalization maps ${\pi_F \colon X^\N \to X^F}$ where $F$ is a finite subset of $\N$ as before.
		That is, we need to prove
		\begin{equation}\label{eq:generalized_spreadability_proof1}
			\begin{tikzpicture}
	\begin{pgfonlayer}{nodelayer}
		\node [style=medium box] (0) at (-3.375, -1.25) {$\quad p \quad$};
		\node [style=morphism] (1) at (-2.5, 0.375) {$X^i$};
		\node [style=none] (2) at (-2.5, -1.25) {};
		\node [style=none] (3) at (-4.25, -1.25) {};
		\node [style=none] (5) at (-3.375, -2.5) {};
		\node [style=none] (6) at (-4.25, 2.5) {};
		\node [style=medium box] (7) at (2.875, -1.25) {$\quad p \quad$};
		\node [style=none] (8) at (3.75, -1.25) {};
		\node [style=none] (9) at (2, -1.25) {};
		\node [style=none] (10) at (2.875, -2.5) {};
		\node [style=none] (12) at (2, 2.5) {};
		\node [style=none] (13) at (0, 0) {$=$};
		\node [style=morphism] (14) at (-2.5, 1.75) {$\pi_F$};
		\node [style=morphism] (15) at (3.75, 1) {$\pi_F$};
		\node [style=none] (16) at (3.75, 2.5) {};
		\node [style=none] (17) at (-2.5, 2.5) {};
	\end{pgfonlayer}
	\begin{pgfonlayer}{edgelayer}
		\draw [style=double wire] (2.center) to (1);
		\draw (5.center) to (0);
		\draw (3.center) to (6.center);
		\draw (10.center) to (7);
		\draw (12.center) to (9.center);
		\draw (17.center) to (14);
		\draw (16.center) to (15);
		\draw [style=double wire] (15) to (8.center);
		\draw [style=double wire] (14) to (1);
	\end{pgfonlayer}
\end{tikzpicture}
		\end{equation}
		for each finite $F \subseteq \N$.
		For any given $i$ and $F$, we can find a finite permutation $\sigma \colon \N \to \N$ whose action on $F$ coincides with that of $i$.
		By definitions of $X^\sigma$ and $X^i$, this implies $\pi_F \circ X^\sigma = \pi_F \circ X^i$, and the claim now follows by the assumed finite permutation invariance.
	\end{proof}

	The next lemma can be thought of as a synthetic statement about stochastic dynamical systems where $t \colon \Theta \to \Theta$ is the dynamics, $p$ is an invariant probability measure, and $f$ is an observable.
	
	\begin{lemma}[dynamically invariant observables]
		\label{lem:invariant_observable}
		Let $p \colon I \to \Theta$ and $t \colon \Theta \to \Theta$ be two morphisms satisfying $t \, p = p$. 
		If $t^\dagger$ denotes a Bayesian inverse of $t$ with respect to $p$, then every \as{$p$} deterministic $f \colon \Theta \to X$ satisfies
		\begin{equation}\label{eq:invariant_observable}
			\begin{tikzpicture}
	\begin{pgfonlayer}{nodelayer}
		\node [style=morphism] (0) at (-7.25, 0.75) {$f$};
		\node [style=morphism] (1) at (-7.25, -0.75) {$t^\dag$};
		\node [style=morphism] (2) at (-3, 0) {$f$};
		\node [style=morphism] (3) at (3, 0.75) {$f$};
		\node [style=morphism] (4) at (3, -0.75) {$t$};
		\node [style=morphism] (5) at (7.25, 0) {$f$};
		\node [style=none] (6) at (-5, 0) {$\ase{p}$};
		\node [style=none] (7) at (5.25, 0) {$\ase{p}$};
		\node [style=none] (8) at (0, 0) {$\Longrightarrow$};
		\node [style=none] (9) at (-7.25, 2) {};
		\node [style=none] (10) at (-7.25, -2) {};
		\node [style=none] (11) at (-3, -1.5) {};
		\node [style=none] (12) at (-3, 1.5) {};
		\node [style=none] (13) at (3, 2) {};
		\node [style=none] (14) at (3, -2) {};
		\node [style=none] (15) at (7.25, -1.5) {};
		\node [style=none] (16) at (7.25, 1.5) {};
	\end{pgfonlayer}
	\begin{pgfonlayer}{edgelayer}
		\draw (9.center) to (0);
		\draw (0) to (1);
		\draw (1) to (10.center);
		\draw (11.center) to (2);
		\draw (2) to (12.center);
		\draw (13.center) to (3);
		\draw (3) to (4);
		\draw (4) to (14.center);
		\draw (15.center) to (5);
		\draw (5) to (16.center);
	\end{pgfonlayer}
\end{tikzpicture}
		\end{equation}
	\end{lemma}
	
	Before presenting the proof, let us again instantiate this for the concrete case of $\cat{BorelStoch}$, and in particular for $X=\R$. Consider a standard Borel probability space $(\Theta,p)$ and a measure-preserving Markov kernel $t \colon \Theta\to\Theta$. 
	A Bayesian inverse of $t$ is Markov kernel $t^\dagger$ satisfying
	\begin{equation}
		\int_{S} t(T|\theta)\, p(d\theta) = \int_{T} t^\dagger ( S |\theta  )  \, p(d\theta)
	\end{equation}
	for all measurable subsets $S,T \subseteq\Theta$.
	Lemma~\ref{lem:invariant_observable} says that if a (deterministic) function $f \colon \Theta \to \R$ satisfies
	\begin{equation}
		f(\theta) = \int_\Theta f(\theta')\,t^\dagger(d\theta'|\theta),
	\end{equation}
	for $p$-almost every $\theta$, then it also satisfies
	\begin{equation}
		f(\theta) = \int_\Theta f(\theta')\,t(d\theta'|\theta).
	\end{equation}
	for $p$-almost every $\theta$.
	By Lemma~\ref{lem:invariant_observable}, if $f$ is an observable invariant under the time-reversed dynamics, then it is also preserved forward in time.
	One may therefore expect this result to have further relevance for a synthetic treatment of ergodic theory.
	
	\begin{proof}
		Using the definition of Bayesian inverses, the assumed $t \, p = p$, and the antecedent statement in implication \eqref{eq:invariant_observable}, we get
		\begin{equation}\label{eq:invariant_observable_proof}
			\begin{tikzpicture}
	\begin{pgfonlayer}{nodelayer}
		\node [style=morphism] (0) at (-5, 1.75) {$f$};
		\node [style=morphism] (1) at (-5, 0.25) {$t$};
		\node [style=none] (2) at (-5, 2.75) {};
		\node [style=none] (3) at (-5, 0) {};
		\node [style=bn] (4) at (-4, -1) {};
		\node [style=none] (5) at (-3, 0) {};
		\node [style=none] (6) at (-3, 2.75) {};
		\node [style=state] (7) at (-4, -2) {$p$};
		\node [style=morphism] (8) at (-3, 1.75) {$f$};
		\node [style=none] (9) at (-1, 0) {$=$};
		\node [style=morphism] (10) at (1, 1.75) {$f$};
		\node [style=none] (12) at (1, 2.75) {};
		\node [style=none] (13) at (1, 0) {};
		\node [style=bn] (14) at (2, -1) {};
		\node [style=none] (15) at (3, 0) {};
		\node [style=none] (16) at (3, 2.75) {};
		\node [style=state] (17) at (2, -3.25) {$p$};
		\node [style=morphism] (18) at (3, 1.75) {$f$};
		\node [style=morphism] (19) at (3, 0.25) {$t^\dag$};
		\node [style=morphism] (20) at (2, -2) {$t$};
		\node [style=none] (21) at (5, 0) {$=$};
		\node [style=morphism] (22) at (7, 1.5) {$f$};
		\node [style=none] (23) at (7, 2.75) {};
		\node [style=none] (24) at (7, 0) {};
		\node [style=bn] (25) at (8, -1) {};
		\node [style=none] (26) at (9, 0) {};
		\node [style=none] (27) at (9, 2.75) {};
		\node [style=state] (28) at (8, -2) {$p$};
		\node [style=morphism] (29) at (9, 1.5) {$f$};
	\end{pgfonlayer}
	\begin{pgfonlayer}{edgelayer}
		\draw (2.center) to (0);
		\draw (0) to (1);
		\draw (1) to (3.center);
		\draw [in=165, out=-90] (3.center) to (4);
		\draw [in=-90, out=15] (4) to (5.center);
		\draw (7) to (4);
		\draw (8) to (5.center);
		\draw (8) to (6.center);
		\draw (12.center) to (10);
		\draw [in=165, out=-90] (13.center) to (14);
		\draw [in=-90, out=15] (14) to (15.center);
		\draw (18) to (16.center);
		\draw (10) to (13.center);
		\draw (19) to (15.center);
		\draw (19) to (18);
		\draw (20) to (17);
		\draw (20) to (14);
		\draw (23.center) to (22);
		\draw [in=165, out=-90] (24.center) to (25);
		\draw [in=-90, out=15] (25) to (26.center);
		\draw (29) to (27.center);
		\draw (22) to (24.center);
		\draw (29) to (26.center);
		\draw (25) to (28);
	\end{pgfonlayer}
\end{tikzpicture}
		\end{equation}
		The claim now follows by~\cite[Lemma~5.5]{fritzrischel2019zeroone} whose proof, even if originally stated for deterministic $f$ only, works just as well for \as{$p$} deterministic $f$.
	\end{proof}
	
	We can present the proof of Proposition~\ref{prop:tail_algebra} now.
	It is worth noting that neither Lemma~\ref{lem:spreadability} nor Lemma~\ref{lem:invariant_observable} use the assumption of representability of $\cC$. 
	As expressed in the statement of Proposition~\ref{prop:tail_algebra}, the following proof does not require representability either, even while making use of the previous two lemmas.
	
	\begin{proof}[Proof of Proposition~\ref{prop:tail_algebra}]
		Let us argue that shift invariance as expressed by equation \eqref{eq:shift_invariance} implies finite permutation invariance as expressed by equation \eqref{eq:perm_invariance}.
		Indeed, for any finite permutation $\sigma \colon \N \to \N$, there is an integer $n_{\sigma}$ beyond which all larger integers are fixed by $\sigma$.
		With $s \colon \N \to \N$ denoting the successor function introduced in the statement of Proposition~\ref{prop:tail_algebra}, this entails $\sigma \circ s^{n_\sigma} = s^{n_\sigma}$.
		Therefore, we have
		\begin{equation}\label{eq:delete_perm}
			(X^s)^{n_{\sigma}} \circ X^{\sigma} = (X^s)^{n_{\sigma}},
		\end{equation}
		where the morphism on the right-hand side effectively discards the first $n_{\sigma}$ components of $X^\N$.
		Thus, assuming shift invariance of $f$, we can infer the permutation invariance as follows:
		\begin{equation}\label{eq:perm_from_shift}
			\begin{tikzpicture}
	\begin{pgfonlayer}{nodelayer}
		\node [style=none] (6) at (0, 4) {$=$};
		\node [style=none] (18) at (7.75, 4) {$=$};
		\node [style=bn] (31) at (-3.75, 3) {};
		\node [style=none] (32) at (-5, 4.25) {};
		\node [style=none] (33) at (-2.5, 4.25) {};
		\node [style=morphism] (34) at (-5, 6.25) {$f$};
		\node [style=none] (35) at (-2.5, 7.25) {};
		\node [style=none] (36) at (-5, 7.25) {};
		\node [style=bn] (38) at (3.75, 3) {};
		\node [style=none] (39) at (2.5, 4.25) {};
		\node [style=none] (40) at (5, 4.25) {};
		\node [style=morphism] (41) at (2.5, 5.25) {$f$};
		\node [style=none] (42) at (5, 7.25) {};
		\node [style=none] (43) at (2.5, 7.25) {};
		\node [style=morphism] (44) at (5, 5.25) {$X^{\sigma^{-1}}$};
		\node [style=morphism] (45) at (-5, 4.5) {$X^{\sigma}$};
		\node [style=state] (46) at (-3.75, 2) {$p$};
		\node [style=state] (47) at (3.75, 2) {$p$};
		\node [style=none] (49) at (0, -4.25) {$=$};
		\node [style=bn] (57) at (12, 3) {};
		\node [style=none] (58) at (10.75, 4.25) {};
		\node [style=none] (59) at (13.25, 4.25) {};
		\node [style=morphism] (60) at (10.75, 6.25) {$f$};
		\node [style=none] (61) at (13.25, 7.25) {};
		\node [style=none] (62) at (10.75, 7.25) {};
		\node [style=morphism] (63) at (10.75, 4.5) {$\! (X^{s})^{n_{\sigma}} \!$};
		\node [style=state] (64) at (12, 2) {$p$};
		\node [style=morphism] (65) at (13.25, 6.25) {$X^{\sigma^{-1}}$};
		\node [style=none] (66) at (7.75, -4.25) {$=$};
		\node [style=bn] (67) at (12, -5.25) {};
		\node [style=none] (68) at (10.75, -4) {};
		\node [style=none] (69) at (13.25, -4) {};
		\node [style=morphism] (70) at (10.75, -2.75) {$f$};
		\node [style=none] (71) at (13.25, -1) {};
		\node [style=none] (72) at (10.75, -1) {};
		\node [style=state] (73) at (12, -6.25) {$p$};
		\node [style=bn] (74) at (3.75, -5.25) {};
		\node [style=none] (75) at (2.5, -4) {};
		\node [style=none] (76) at (5, -4) {};
		\node [style=morphism] (77) at (2.5, -2) {$f$};
		\node [style=none] (78) at (5, -1) {};
		\node [style=none] (79) at (2.5, -1) {};
		\node [style=morphism] (80) at (2.5, -3.75) {$\! (X^{s})^{n_{\sigma}} \!$};
		\node [style=state] (81) at (3.75, -6.25) {$p$};
	\end{pgfonlayer}
	\begin{pgfonlayer}{edgelayer}
		\draw (36.center) to (34);
		\draw [style=double wire, in=165, out=-90] (32.center) to (31);
		\draw [style=double wire, in=-90, out=15] (31) to (33.center);
		\draw [style=double wire] (33.center) to (35.center);
		\draw (43.center) to (41);
		\draw [style=double wire, in=165, out=-90] (39.center) to (38);
		\draw [style=double wire, in=-90, out=15] (38) to (40.center);
		\draw (32.center) to (45);
		\draw [style=double wire] (47) to (38);
		\draw [style=double wire] (46) to (31);
		\draw (62.center) to (60);
		\draw [style=double wire, in=165, out=-90] (58.center) to (57);
		\draw [style=double wire, in=-90, out=15] (57) to (59.center);
		\draw (58.center) to (63);
		\draw [style=double wire] (64) to (57);
		\draw [style=double wire] (59.center) to (65);
		\draw [style=double wire] (65) to (61.center);
		\draw (72.center) to (70);
		\draw [style=double wire, in=165, out=-90] (68.center) to (67);
		\draw [style=double wire, in=-90, out=15] (67) to (69.center);
		\draw [style=double wire] (69.center) to (71.center);
		\draw [style=double wire] (73) to (67);
		\draw (79.center) to (77);
		\draw [style=double wire, in=165, out=-90] (75.center) to (74);
		\draw [style=double wire, in=-90, out=15] (74) to (76.center);
		\draw [style=double wire] (76.center) to (78.center);
		\draw (75.center) to (80);
		\draw [style=double wire] (81) to (74);
		\draw [style=double wire] (39.center) to (41);
		\draw [style=double wire] (40.center) to (44);
		\draw [style=double wire] (68.center) to (70);
		\draw [style=double wire] (80) to (77);
		\draw [style=double wire] (45) to (34);
		\draw [style=double wire] (63) to (60);
		\draw [style=double wire] (44) to (42.center);
	\end{pgfonlayer}
\end{tikzpicture}
		\end{equation}
		In particular, the first equality holds by exchangeability of $p$ and because $X^{\sigma^{-1}}$ is a Bayesian inverse of $X^\sigma$.
		The latter holds by virtue of the fact that $X^{\sigma^{-1}}$ is deterministic and the inverse of $X^\sigma$.
		The second and fourth equations in \eqref{eq:perm_from_shift} correspond to shift invariance of $f$ applied $n_\sigma$ times, while the third is analogous to the first in reverse, using equation \eqref{eq:delete_perm} in addition.
		
		Conversely, let us now assume finite permutation invariance of $f$ and show that it implies shift invariance. 
		By the first equality of \eqref{eq:perm_from_shift}, we have that the morphism
		
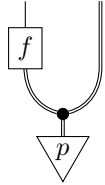
\begin{equation*}\label{eq:shift_from_perm1}
			\begin{tikzpicture}
	\begin{pgfonlayer}{nodelayer}
		\node [style=bn] (0) at (0, -0.75) {};
		\node [style=state] (1) at (0, -1.75) {$p$};
		\node [style=morphism] (2) at (-1, 1) {$f$};
		\node [style=none] (3) at (-1, 0.5) {};
		\node [style=none] (4) at (-1, 2.25) {};
		\node [style=none] (5) at (1, 0.5) {};
		\node [style=none] (6) at (1, 2.25) {};
	\end{pgfonlayer}
	\begin{pgfonlayer}{edgelayer}
		\draw (4.center) to (2);
		\draw [style=double wire] (2) to (3.center);
		\draw [style=double wire, in=165, out=-90] (3.center) to (0);
		\draw [style=double wire] (0) to (1);
		\draw [style=double wire, in=-90, out=15] (0) to (5.center);
		\draw [style=double wire] (6.center) to (5.center);
	\end{pgfonlayer}
\end{tikzpicture}
		\end{equation*}
		is exchangeable in the second output.
		Therefore, by Lemma~\ref{lem:spreadability}, we also obtain
		\begin{equation}\label{eq:shift_from_perm2}
			\begin{tikzpicture}
	\begin{pgfonlayer}{nodelayer}
		\node [style=bn] (0) at (-3.5, -0.75) {};
		\node [style=state] (1) at (-3.5, -1.75) {$p$};
		\node [style=morphism] (2) at (-4.5, 1.5) {$f$};
		\node [style=none] (3) at (-4.5, 0.5) {};
		\node [style=none] (4) at (-4.5, 3) {};
		\node [style=none] (5) at (-2.5, 0.5) {};
		\node [style=none] (6) at (-2.5, 3) {};
		\node [style=none] (7) at (0, 0) {$=$};
		\node [style=morphism] (8) at (-2.5, 1.5) {$X^s$};
		\node [style=bn] (9) at (4, -1.5) {};
		\node [style=state] (10) at (4, -2.5) {$p$};
		\node [style=morphism] (11) at (3, 1.75) {$f$};
		\node [style=none] (12) at (3, -0.5) {};
		\node [style=none] (13) at (3, 3) {};
		\node [style=none] (14) at (5, -0.5) {};
		\node [style=none] (15) at (5, 3) {};
		\node [style=none] (16) at (-7, 0) {$=$};
		\node [style=bn] (17) at (-10, -0.75) {};
		\node [style=state] (18) at (-10, -1.75) {$p$};
		\node [style=morphism] (19) at (-11, 1.5) {$f$};
		\node [style=none] (20) at (-11, 0.5) {};
		\node [style=none] (21) at (-11, 3) {};
		\node [style=none] (22) at (-9, 0.5) {};
		\node [style=none] (23) at (-9, 3) {};
		\node [style=morphism] (24) at (3, 0) {$(X^s)^\dag$};
	\end{pgfonlayer}
	\begin{pgfonlayer}{edgelayer}
		\draw (4.center) to (2);
		\draw [style=double wire] (2) to (3.center);
		\draw [style=double wire, in=165, out=-90] (3.center) to (0);
		\draw [style=double wire] (0) to (1);
		\draw [style=double wire, in=-90, out=15] (0) to (5.center);
		\draw [style=double wire] (5.center) to (8);
		\draw [style=double wire] (8) to (6.center);
		\draw (13.center) to (11);
		\draw [style=double wire, in=165, out=-90] (12.center) to (9);
		\draw [style=double wire] (9) to (10);
		\draw [style=double wire, in=-90, out=15] (9) to (14.center);
		\draw [style=double wire] (15.center) to (14.center);
		\draw (21.center) to (19);
		\draw [style=double wire] (19) to (20.center);
		\draw [style=double wire, in=165, out=-90] (20.center) to (17);
		\draw [style=double wire] (17) to (18);
		\draw [style=double wire, in=-90, out=15] (17) to (22.center);
		\draw [style=double wire] (23.center) to (22.center);
		\draw [style=double wire] (11) to (24);
		\draw [style=double wire] (12.center) to (24);
	\end{pgfonlayer}
\end{tikzpicture}
		\end{equation}
		By Lemma~\ref{lem:invariant_observable}, this entails
		\begin{equation}\label{eq:shift_from_perm3}
			\begin{tikzpicture}
	\begin{pgfonlayer}{nodelayer}
		\node [style=medium box] (0) at (-3, 0) {$f$};
		\node [style=none] (3) at (-3, 1.5) {};
		\node [style=none] (4) at (0, 0) {$\ase{p}$};
		\node [style=morphism] (5) at (3.25, 0.75) {$f$};
		\node [style=none] (7) at (3.25, 2) {};
		\node [style=none] (8) at (3.25, -2.25) {};
		\node [style=none] (11) at (-3, -1.5) {};
		\node [style=morphism] (12) at (3.25, -1.25) {$X^s$};
	\end{pgfonlayer}
	\begin{pgfonlayer}{edgelayer}
		\draw (7.center) to (5);
		\draw (3.center) to (0);
		\draw [style=double wire] (12) to (8.center);
		\draw [style=double wire] (12) to (5);
		\draw [style=double wire] (0) to (11.center);
	\end{pgfonlayer}
\end{tikzpicture}
		\end{equation}
		which is what we wanted to prove.
	\end{proof}
	
	For a given exchangeable morphism $p \colon I \to X^\N$, we now consider a conditional $\tailcond{p} \colon X^\N \to X$ of the first output given all others, which means that it satisfies the left equation in
	\begin{equation}\label{eq:tail_conditional_def}
		\begin{tikzpicture}
	\begin{pgfonlayer}{nodelayer}
		\node [style=none] (6) at (-3.25, 0) {$=$};
		\node [style=bn] (11) at (0.25, 1) {};
		\node [style=none] (13) at (-0.75, 2.25) {};
		\node [style=none] (14) at (1.25, 2.25) {};
		\node [style=morphism] (15) at (-0.75, 2.5) {$\tailcond{p}$};
		\node [style=none] (16) at (1.25, 3.5) {};
		\node [style=none] (17) at (-0.75, 3.5) {};
		\node [style=none] (18) at (3.25, 0) {$=$};
		\node [style=bn] (23) at (6.75, -0.25) {};
		\node [style=none] (25) at (5.75, 1) {};
		\node [style=none] (26) at (7.75, 1) {};
		\node [style=morphism] (27) at (5.75, 1.25) {$\tailcond{p}$};
		\node [style=none] (28) at (7.75, 2.25) {};
		\node [style=none] (29) at (5.75, 2.25) {};
		\node [style=none] (31) at (-6, -0.25) {};
		\node [style=none] (32) at (-6.5, 1.25) {};
		\node [style=none] (33) at (-5.5, 1.25) {};
		\node [style=none] (34) at (-6, 0.25) {};
		\node [style=none] (35) at (-6.035, 0.25) {};
		\node [style=none] (36) at (-5.5, 2) {};
		\node [style=none] (37) at (-6.5, 2) {};
		\node [style=none] (39) at (-0.25, -1.5) {};
		\node [style=none] (40) at (-0.75, 0) {};
		\node [style=none] (41) at (0.25, 0) {};
		\node [style=none] (42) at (-0.25, -1) {};
		\node [style=none] (43) at (-0.285, -1) {};
		\node [style=bn] (45) at (-0.75, 0.25) {};
		\node [style=state] (46) at (6.75, -1.25) {$p$};
		\node [style=state] (47) at (-0.25, -2) {$p$};
		\node [style=state] (48) at (-6, -0.75) {$p$};
	\end{pgfonlayer}
	\begin{pgfonlayer}{edgelayer}
		\draw (17.center) to (15);
		\draw [style=double wire, in=165, out=-90] (13.center) to (11);
		\draw [style=double wire, in=-90, out=15] (11) to (14.center);
		\draw [style=double wire] (14.center) to (16.center);
		\draw (29.center) to (27);
		\draw [style=double wire, in=165, out=-90] (25.center) to (23);
		\draw [style=double wire, in=-90, out=15] (23) to (26.center);
		\draw [style=double wire] (26.center) to (28.center);
		\draw [style=double wire] (31.center) to (34.center);
		\draw [style=double wire, in=-90, out=90] (34.center) to (33.center);
		\draw [in=-90, out=90] (35.center) to (32.center);
		\draw [style=double wire] (33.center) to (36.center);
		\draw (32.center) to (37.center);
		\draw [style=double wire] (39.center) to (42.center);
		\draw [style=double wire, in=-90, out=90] (42.center) to (41.center);
		\draw [in=-90, out=90] (43.center) to (40.center);
		\draw (40.center) to (45);
		\draw [style=double wire] (41.center) to (11);
		\draw [style=double wire] (46) to (23);
		\draw [style=double wire] (47) to (39.center);
		\draw [style=double wire] (48) to (31.center);
	\end{pgfonlayer}
\end{tikzpicture}
	\end{equation}
	while the right one follows by exchangeability of $p$ and Lemma~\ref{lem:spreadability}.
	
	Such a tail conditional depends only on the tail of the product $X^\N$ in the following sense.
	
	\begin{lemma}[shift invariance of the tail conditional]
		\label{lem:tail_is_tail}
		If $p \colon I \to X^\N$ is exchangeable, then we have
		\begin{equation}\label{eq:tail_conditional_shift}
			\begin{tikzpicture}
	\begin{pgfonlayer}{nodelayer}
		\node [style=medium box] (0) at (-3, 0) {$\tailcond{p}$};
		\node [style=none] (3) at (-3, 1.5) {};
		\node [style=none] (4) at (0, 0) {$\ase{p}$};
		\node [style=morphism] (5) at (3.75, 0.75) {$\tailcond{p}$};
		\node [style=none] (7) at (3.75, 2) {};
		\node [style=none] (11) at (-3, -1.5) {};
		\node [style=bn] (12) at (2.25, 0.25) {};
		\node [style=none] (13) at (3, -2) {};
		\node [style=none] (14) at (2.25, -0.25) {};
		\node [style=none] (15) at (3.75, -0.25) {};
		\node [style=none] (16) at (3, -1.25) {};
		\node [style=none] (17) at (2.965, -1.25) {};
	\end{pgfonlayer}
	\begin{pgfonlayer}{edgelayer}
		\draw (7.center) to (5);
		\draw (3.center) to (0);
		\draw [style=double wire] (11.center) to (0);
		\draw [style=double wire] (13.center) to (16.center);
		\draw [style=double wire, in=-90, out=90] (16.center) to (15.center);
		\draw [in=-90, out=90] (17.center) to (14.center);
		\draw (14.center) to (12);
		\draw [style=double wire] (15.center) to (5);
	\end{pgfonlayer}
\end{tikzpicture}
		\end{equation}
	\end{lemma}
	
	Here is what the statement looks like in $\cat{BorelStoch}$. 
	Let $p$ be an exchangeable probability measure on $X^\N$, and let $\tailcond{p}$ be a (regular) conditional of the first component given the other ones. 
	Denote a generic element of $X^\N$ by $(x_1,x_2,\dots)$. 
	Then for each measurable subset $S \subseteq X$, we have that 
	\begin{equation}
		\tailcond{p}(S|x_1,x_2,\dots) = \tailcond{p}(S|x_2,x_3,\dots)
	\end{equation}
	holds for $p$-almost all sequences $(x_1,x_2,\dots)\in X^\N$. Iterating this equation shows that $\tailcond{p}$ only depends on the tail.
	
	\begin{proof}
		By the definition of $\tailcond{p}$ and exchangeability of $p$, we have
		\begin{equation}\label{eq:tail_conditional_exc}
			\begin{tikzpicture}
	\begin{pgfonlayer}{nodelayer}
		\node [style=none] (6) at (0, 0) {$=$};
		\node [style=none] (18) at (7, 0) {$=$};
		\node [style=bn] (23) at (10.75, -0.25) {};
		\node [style=none] (25) at (9.5, 1) {};
		\node [style=none] (26) at (12, 1) {};
		\node [style=morphism] (27) at (9.5, 1.5) {$\tailcond{p}$};
		\node [style=none] (28) at (12, 3) {};
		\node [style=none] (29) at (9.5, 3) {};
		\node [style=bn] (31) at (-3.5, -1) {};
		\node [style=none] (32) at (-4.75, 0.25) {};
		\node [style=none] (33) at (-2.25, 0.25) {};
		\node [style=morphism] (34) at (-4.75, 2) {$\tailcond{p}$};
		\node [style=none] (35) at (-2.25, 3) {};
		\node [style=none] (36) at (-4.75, 3) {};
		\node [style=bn] (38) at (3.5, -0.25) {};
		\node [style=none] (39) at (2.25, 1) {};
		\node [style=none] (40) at (4.75, 1) {};
		\node [style=morphism] (41) at (2.25, 1.5) {$\tailcond{p}$};
		\node [style=none] (42) at (4.75, 3) {};
		\node [style=none] (43) at (2.25, 3) {};
		\node [style=morphism] (44) at (4.75, 1.5) {$X^{\sigma^{-1}}$};
		\node [style=morphism] (45) at (-4.75, 0.5) {$X^{\sigma}$};
		\node [style=state] (46) at (-3.5, -2) {$p$};
		\node [style=state] (47) at (3.5, -1.25) {$p$};
		\node [style=state] (48) at (10.75, -1.25) {$p$};
	\end{pgfonlayer}
	\begin{pgfonlayer}{edgelayer}
		\draw (29.center) to (27);
		\draw [style=double wire, in=165, out=-90] (25.center) to (23);
		\draw [style=double wire, in=-90, out=15] (23) to (26.center);
		\draw [style=double wire] (26.center) to (28.center);
		\draw (36.center) to (34);
		\draw [style=double wire, in=165, out=-90] (32.center) to (31);
		\draw [style=double wire, in=-90, out=15] (31) to (33.center);
		\draw [style=double wire] (33.center) to (35.center);
		\draw (43.center) to (41);
		\draw [style=double wire, in=165, out=-90] (39.center) to (38);
		\draw [style=double wire, in=-90, out=15] (38) to (40.center);
		\draw (40.center) to (44);
		\draw (44) to (42.center);
		\draw (32.center) to (45);
		\draw [style=double wire] (45) to (34);
		\draw (39.center) to (41);
		\draw (25.center) to (27);
		\draw [style=double wire] (48) to (23);
		\draw [style=double wire] (47) to (38);
		\draw [style=double wire] (46) to (31);
	\end{pgfonlayer}
\end{tikzpicture}
		\end{equation}
		for any finite permutation $\sigma \colon \N \to \N$.
		By the $\as{}$-compatible representability of $\cC$ (Definition~\ref{def:as_comp_rep}), this is equivalent to 
		\begin{equation}
			\left( \tailcond{p} \circ X^{\sigma} \right)^\sharp \;\ase{p}\; \tailcond{p}^\sharp.
		\end{equation}
		Since $X^{\sigma}$ is deterministic, we have 
		\begin{equation}
			\left( \tailcond{p} \circ X^{\sigma} \right)^\sharp \;=\; \tailcond{p}^\sharp \circ X^{\sigma},
		\end{equation}
		which is an instance of the naturality of~\eqref{eq:representability},
		and conclude that the deterministic morphism $\tailcond{p}^\sharp$ is $p$-almost surely finite permutation invariant.
		Thus, by Proposition~\ref{prop:tail_algebra}, it is also shift invariant.
		This proves the claim after composition with the sampling map $\samp_X$.
	\end{proof}
	
	\begin{lemma}[exchangeable states are conditionally iid]
		\label{lem:conditionally_iid}
		If $p \colon I \to X^\N$ is an exchangeable morphism, then we have
		\begin{equation}\label{eq:conditionally_iid}
			\begin{tikzpicture}
	\begin{pgfonlayer}{nodelayer}
		\node [style=none] (0) at (0, 0) {$=$};
		\node [style=state] (1) at (-2.5, -0.25) {$p$};
		\node [style=none] (2) at (-2.5, 1) {};
		\node [style=state] (3) at (4.25, -2) {$p$};
		\node [style=bn] (4) at (4.25, -1) {};
		\node [style=morphism] (5) at (2.5, 0.5) {$\tailcond{p}$};
		\node [style=morphism] (6) at (6, 0.5) {$\tailcond{p}$};
		\node [style=none] (7) at (2.5, 0.25) {};
		\node [style=none] (8) at (6, 0.25) {};
		\node [style=none] (9) at (2.5, 1.75) {};
		\node [style=none] (10) at (6, 1.75) {};
		\node [style=none] (11) at (4.25, 0.5) {${\cdots}$};
	\end{pgfonlayer}
	\begin{pgfonlayer}{edgelayer}
		\draw [style=double wire] (2.center) to (1);
		\draw [style=double wire] (3) to (4);
		\draw [style=double wire, in=-90, out=15] (4) to (8.center);
		\draw [style=double wire, in=-90, out=165] (4) to (7.center);
		\draw (10.center) to (6);
		\draw (9.center) to (5);
		\draw [style=double wire] (6) to (8.center);
		\draw [style=double wire] (7.center) to (5);
	\end{pgfonlayer}
\end{tikzpicture}
		\end{equation}
	\end{lemma}
	
	Concretely, in $\cat{BorelStoch}$, Lemma~\ref{lem:conditionally_iid} says the following. 
	Let $p$ be an exchangeable probability measure on $X^\N$. 
	Then for every cylinder defined by measurable subsets $S_1,\dots, S_N\subseteq X$, we have  
	\begin{equation}
		p(S_1\times\dots\times S_n) = \int_{X^\N} \tailcond{p}(S_1|\xi) \times \cdots \times \tailcond{p}(S_n|\xi) \, p(d\xi) ,
	\end{equation}
	where $\xi$ is shorthand for a generic sequence $(x_1,x_2,\dots)\in X^\N$. 
	
	\begin{proof}
		By the bijection of Definition~\ref{def:inf_product}, it is enough to prove the statement for all finite products, i.e.~to prove that
		\begin{equation}\label{eq:conditionally_iid_proof1}
			\begin{tikzpicture}
	\begin{pgfonlayer}{nodelayer}
		\node [style=state] (0) at (-2.5, -0.75) {$p$};
		\node [style=none] (1) at (-2.5, 0.75) {};
		\node [style=none] (2) at (0, 0) {$=$};
		\node [style=state] (3) at (6, -2) {$p$};
		\node [style=bn] (4) at (6, -1) {};
		\node [style=none] (5) at (2.5, 0.5) {};
		\node [style=none] (6) at (6, 0.5) {};
		\node [style=none] (7) at (7.5, 0.5) {};
		\node [style=none] (8) at (4.25, 0.75) {${\cdots}$};
		\node [style=none] (9) at (2.5, 2) {};
		\node [style=none] (10) at (6, 2) {};
		\node [style=none] (11) at (7.5, 2) {};
		\node [style=morphism] (12) at (2.5, 0.75) {$\tailcond{p}$};
		\node [style=morphism] (13) at (6, 0.75) {$\tailcond{p}$};
		\node [style=none] (14) at (4.25, 3) {$n$ wires};
		\node [style=none] (15) at (2.375, 2.2) {};
		\node [style=none] (16) at (6.125, 2.2) {};
	\end{pgfonlayer}
	\begin{pgfonlayer}{edgelayer}
		\draw [style=double wire] (0) to (1.center);
		\draw [style=double wire] (3) to (4);
		\draw [style=double wire, in=-90, out=-180, looseness=0.75] (4) to (5.center);
		\draw [style=double wire] (4) to (6.center);
		\draw [style=double wire, in=-90, out=15] (4) to (7.center);
		\draw [style=double wire] (7.center) to (11.center);
		\draw [style=double wire] (6.center) to (13);
		\draw [style=double wire] (5.center) to (12);
		\draw (12) to (9.center);
		\draw (13) to (10.center);
		\draw [style=curly brace] (15.center) to (16.center);
	\end{pgfonlayer}
\end{tikzpicture}
		\end{equation}
		for every $n \in \N$. 
		Using induction, the base case $n = 0$ is trivial. 
		In order to get from $n$ to $n+1$, we use
		\begin{equation*}\label{eq:conditionally_iid_proof2}
			\begin{tikzpicture}
	\begin{pgfonlayer}{nodelayer}
		\node [style=state] (0) at (-5.75, 2) {$p$};
		\node [style=bn] (1) at (-5.75, 3.75) {};
		\node [style=none] (2) at (-8.75, 5.25) {};
		\node [style=none] (3) at (-5.75, 5.25) {};
		\node [style=none] (4) at (-4, 5.25) {};
		\node [style=none] (5) at (-7.25, 5.75) {$\mydots$};
		\node [style=none] (6) at (-8.75, 7.5) {};
		\node [style=none] (7) at (-5.75, 7.5) {};
		\node [style=morphism] (9) at (-8.75, 5.75) {$\tailcond{p}$};
		\node [style=morphism] (10) at (-5.75, 5.75) {$\tailcond{p}$};
		\node [style=none] (12) at (-2, 4.75) {$=$};
		\node [style=state] (13) at (3.75, 2) {$p$};
		\node [style=bn] (14) at (3.75, 3) {};
		\node [style=none] (15) at (0.75, 4.5) {};
		\node [style=none] (16) at (3.75, 4.5) {};
		\node [style=none] (17) at (5.75, 4.5) {};
		\node [style=none] (18) at (2.25, 5.5) {$\mydots$};
		\node [style=none] (19) at (0.75, 7.5) {};
		\node [style=none] (20) at (3.75, 7.5) {};
		\node [style=morphism] (22) at (0.75, 6.25) {$\tailcond{p}$};
		\node [style=morphism] (23) at (3.75, 6.25) {$\tailcond{p}$};
		\node [style=morphism] (25) at (0.75, 4.75) {$X^s$};
		\node [style=morphism] (26) at (3.75, 4.75) {$X^s$};
		\node [style=state] (27) at (13.5, 1) {$p$};
		\node [style=bn] (28) at (13.75, 4.25) {};
		\node [style=none] (29) at (10.75, 5.75) {};
		\node [style=none] (30) at (13.75, 5.75) {};
		\node [style=none] (31) at (16, 5.75) {};
		\node [style=none] (32) at (12.25, 6.25) {$\mydots$};
		\node [style=none] (33) at (10.75, 7.5) {};
		\node [style=none] (34) at (13.75, 7.5) {};
		\node [style=morphism] (36) at (10.75, 6.25) {$\tailcond{p}$};
		\node [style=morphism] (37) at (13.75, 6.25) {$\tailcond{p}$};
		\node [style=none] (41) at (8.25, 4.75) {$=$};
		\node [style=none] (46) at (15.25, 4.925) {};
		\node [style=none] (47) at (-2, -4.75) {$=$};
		\node [style=none] (79) at (8.25, -4.75) {$=$};
		\node [style=none] (82) at (15.25, 7.5) {};
		\node [style=none] (83) at (16, 7.5) {};
		\node [style=none] (86) at (13.5, 1.5) {};
		\node [style=none] (87) at (13.25, 3) {};
		\node [style=none] (88) at (13.75, 3) {};
		\node [style=none] (89) at (13.5, 2) {};
		\node [style=none] (90) at (13.465, 2) {};
		\node [style=state] (111) at (13.75, -7.25) {$p$};
		\node [style=bn] (112) at (13.75, -5.75) {};
		\node [style=none] (113) at (10.75, -4.25) {};
		\node [style=none] (114) at (13.75, -4.25) {};
		\node [style=none] (115) at (15.25, -4.25) {};
		\node [style=none] (116) at (12.25, -4) {$\mydots$};
		\node [style=none] (117) at (10.75, -2.5) {};
		\node [style=none] (118) at (13.75, -2.5) {};
		\node [style=none] (119) at (15.25, -2.5) {};
		\node [style=morphism] (120) at (10.75, -4) {$\tailcond{p}$};
		\node [style=morphism] (121) at (13.75, -4) {$\tailcond{p}$};
		\node [style=none] (123) at (15.25, 4.725) {};
		\node [style=none] (124) at (-4, 5.5) {};
		\node [style=none] (125) at (-4.5, 7) {};
		\node [style=none] (126) at (-3.5, 7) {};
		\node [style=none] (127) at (-4, 6) {};
		\node [style=none] (128) at (-4.035, 6) {};
		\node [style=none] (130) at (5.75, 5.5) {};
		\node [style=none] (131) at (5.25, 7) {};
		\node [style=none] (132) at (6.25, 7) {};
		\node [style=none] (133) at (5.75, 6) {};
		\node [style=none] (134) at (5.715, 6) {};
		\node [style=none] (135) at (6.25, 7.5) {};
		\node [style=none] (136) at (5.25, 7.5) {};
		\node [style=none] (137) at (13.675, 3.4) {};
		\node [style=none] (138) at (13.825, 3.425) {};
		\node [style=none] (139) at (15.25, 4.375) {};
		\node [style=state] (140) at (3.75, -9.25) {$p$};
		\node [style=bn] (141) at (3.75, -5.25) {};
		\node [style=none] (142) at (0.75, -3.75) {};
		\node [style=none] (143) at (3.75, -3.75) {};
		\node [style=none] (144) at (6, -3.75) {};
		\node [style=none] (145) at (2.25, -3.5) {$\mydots$};
		\node [style=none] (146) at (0.75, -2.25) {};
		\node [style=none] (147) at (3.75, -2.25) {};
		\node [style=morphism] (148) at (0.75, -3.5) {$\tailcond{p}$};
		\node [style=morphism] (149) at (3.75, -3.5) {$\tailcond{p}$};
		\node [style=none] (151) at (5.25, -4.575) {};
		\node [style=none] (152) at (5.25, -2.25) {};
		\node [style=none] (153) at (6, -2.25) {};
		\node [style=none] (157) at (3.75, -6.5) {};
		\node [style=none] (159) at (5.25, -4.775) {};
		\node [style=none] (160) at (3.675, -5.85) {};
		\node [style=none] (161) at (3.825, -5.825) {};
		\node [style=none] (162) at (5.25, -5) {};
		\node [style=bn] (163) at (3.75, -8.25) {};
		\node [style=morphism] (164) at (2.5, -7) {$\tailcond{p}$};
		\node [style=none] (165) at (-4.5, 7.5) {};
		\node [style=none] (166) at (-3.5, 7.5) {};
		\node [style=none] (167) at (2.5, -6.75) {};
		\node [style=none] (168) at (-7.25, 8.5) {$n$ wires};
		\node [style=none] (169) at (-8.875, 7.7) {};
		\node [style=none] (170) at (-5.625, 7.7) {};
		\node [style=none] (171) at (2.25, 8.5) {$n$ wires};
		\node [style=none] (172) at (0.625, 7.7) {};
		\node [style=none] (173) at (3.875, 7.7) {};
		\node [style=none] (174) at (12.25, 8.5) {$n$ wires};
		\node [style=none] (175) at (10.625, 7.7) {};
		\node [style=none] (176) at (13.875, 7.7) {};
		\node [style=none] (177) at (2.25, -1.25) {$n$ wires};
		\node [style=none] (178) at (0.625, -2.05) {};
		\node [style=none] (179) at (3.875, -2.05) {};
		\node [style=none] (180) at (12.25, -1.5) {$n+1$ wires};
		\node [style=none] (181) at (10.625, -2.3) {};
		\node [style=none] (182) at (13.875, -2.3) {};
	\end{pgfonlayer}
	\begin{pgfonlayer}{edgelayer}
		\draw [style=double wire] (0) to (1);
		\draw [style=double wire, in=-90, out=180, looseness=0.75] (1) to (2.center);
		\draw [style=double wire] (1) to (3.center);
		\draw [style=double wire, in=-90, out=15] (1) to (4.center);
		\draw [style=double wire] (3.center) to (10);
		\draw [style=double wire] (2.center) to (9);
		\draw (9) to (6.center);
		\draw (10) to (7.center);
		\draw [style=double wire] (13) to (14);
		\draw [style=double wire, in=-90, out=165, looseness=0.75] (14) to (15.center);
		\draw [style=double wire] (14) to (16.center);
		\draw [style=double wire, in=-90, out=15] (14) to (17.center);
		\draw (22) to (19.center);
		\draw (23) to (20.center);
		\draw [style=double wire] (25) to (15.center);
		\draw [style=double wire] (26) to (16.center);
		\draw [style=double wire] (25) to (22);
		\draw [style=double wire] (23) to (26);
		\draw [style=double wire, in=-90, out=165, looseness=0.75] (28) to (29.center);
		\draw [style=double wire] (28) to (30.center);
		\draw [style=double wire, in=-90, out=15, looseness=0.75] (28) to (31.center);
		\draw (36) to (33.center);
		\draw (37) to (34.center);
		\draw [style=double wire] (36) to (29.center);
		\draw [style=double wire] (37) to (30.center);
		\draw [style=double wire] (86.center) to (89.center);
		\draw [style=double wire, in=-90, out=90] (89.center) to (88.center);
		\draw [in=-90, out=90] (90.center) to (87.center);
		\draw [style=double wire] (86.center) to (27);
		\draw (46.center) to (82.center);
		\draw [style=double wire] (31.center) to (83.center);
		\draw [style=double wire] (111) to (112);
		\draw [style=double wire, in=-90, out=180, looseness=0.75] (112) to (113.center);
		\draw [style=double wire] (112) to (114.center);
		\draw [style=double wire, in=-90, out=15] (112) to (115.center);
		\draw [style=double wire] (115.center) to (119.center);
		\draw [style=double wire] (114.center) to (121);
		\draw [style=double wire] (113.center) to (120);
		\draw (120) to (117.center);
		\draw (121) to (118.center);
		\draw [style=double wire] (124.center) to (127.center);
		\draw [style=double wire, in=-90, out=90] (127.center) to (126.center);
		\draw [in=-90, out=90] (128.center) to (125.center);
		\draw [style=double wire] (4.center) to (124.center);
		\draw [style=double wire] (130.center) to (133.center);
		\draw [style=double wire, in=-90, out=90] (133.center) to (132.center);
		\draw [in=-90, out=90] (134.center) to (131.center);
		\draw [style=double wire] (17.center) to (130.center);
		\draw [style=double wire] (132.center) to (135.center);
		\draw (131.center) to (136.center);
		\draw [style=double wire] (88.center) to (28);
		\draw [in=-90, out=15, looseness=0.75] (138.center) to (139.center);
		\draw (139.center) to (123.center);
		\draw [in=-165, out=90] (87.center) to (137.center);
		\draw [style=double wire, in=-90, out=165, looseness=0.75] (141) to (142.center);
		\draw [style=double wire] (141) to (143.center);
		\draw [style=double wire, in=-90, out=15, looseness=0.75] (141) to (144.center);
		\draw (148) to (146.center);
		\draw (149) to (147.center);
		\draw [style=double wire] (148) to (142.center);
		\draw [style=double wire] (149) to (143.center);
		\draw (151.center) to (152.center);
		\draw [style=double wire] (144.center) to (153.center);
		\draw [in=-90, out=15, looseness=0.75] (161.center) to (162.center);
		\draw (162.center) to (159.center);
		\draw [style=double wire] (140) to (163);
		\draw [style=double wire] (163) to (157.center);
		\draw [style=double wire] (157.center) to (141);
		\draw [style=double wire, in=-90, out=165] (163) to (164);
		\draw (125.center) to (165.center);
		\draw [style=double wire] (126.center) to (166.center);
		\draw [in=-165, out=90, looseness=0.75] (167.center) to (160.center);
		\draw (164) to (167.center);
		\draw [style=curly brace] (169.center) to (170.center);
		\draw [style=curly brace] (172.center) to (173.center);
		\draw [style=curly brace] (175.center) to (176.center);
		\draw [style=curly brace] (178.center) to (179.center);
		\draw [style=curly brace] (181.center) to (182.center);
	\end{pgfonlayer}
\end{tikzpicture}
		\end{equation*}
		where the first step is by Lemma~\ref{lem:tail_is_tail} and the third one by equation \eqref{eq:tail_conditional_def}.
	\end{proof}
	
\smalltitle{Parametrization construction.}	
	
	Lemma~\ref{lem:conditionally_iid} allows us to prove de Finetti's Theorem for exchangeable morphisms out of $I$, but our Theorem~\ref{syntheticthm} is meant to be applicable to exchangeable morphisms with an arbitrary domain $A$.
	In order to make this logical transition, we interpret the statement of Lemma~\ref{lem:conditionally_iid} in a new Markov category $\cC_A$ where every morphism $f \colon A \to Y$ of $\cC$ can be reinterpreted as a morphism $f \colon I \to Y$ in $\cC_A$.
	Specifically, we use the \emph{parametric Markov category} $\cC_A$ \cite[Section 2.2]{fritz2020representable}, which is defined in terms of $\cC$ and $A$ by taking its objects to coincide with those of $\cC$ and its morphisms to be
	\begin{equation}
		\cC_A (B, Y) \coloneqq \cC( A \otimes B , Y).
	\end{equation}
	In a sense, $\cC_A$ contains the same data as $\cC$; it is just organized differently.
	When composing morphisms in $\cC_A$ we use the composition of $\cC$ and distribute a copy of $A$ to each of the morphisms in the composition.
	The key aspect of the parametrization construction for our proof is that as long as $\cC$ satisfies Assumption~\ref{three_assumptions}, so does the parametric category $\cC_A$.
	The arguments for conditionals and \as{}-compatible representability have been presented in \cite{fritz2020representable} as Lemma 2.10 and Lemma 3.22 respectively.
	Here, we also spell out the argument for Kolmogorov powers.
	\begin{lemma}\label{lem:product_in_parametric}
		If $X^\N$ is a countable Kolmogorov power of $X$ in $\cC$, then $X^\N$ is also a countable Kolmogorov power of $X$ in the parametric Markov category $\cC_A$.
	\end{lemma}
	\begin{proof}
		Let $B,Y$ be arbitrary objects of $\cC_A$.
		We have the following three bijective correspondences
		\begin{equation*}
			\begin{tikzcd}
				{f \in \cC_A(B,X^\N \otimes Y)} \arrow[leftrightarrow, r]						&	{f \in \cC(A \otimes B, X^\N \otimes Y)} \arrow[leftrightarrow, d]	\\
				{\left( f_F \in \cC_A(B, X^F \otimes Y) \right)} \arrow[leftrightarrow, r] 	&	{\left( f_F \in \cC(A \otimes B, X^F \otimes Y) \right)}
			\end{tikzcd}
		\end{equation*}
		where the lower two are compatible families in the same sense as in Definition~\ref{def:inf_product}.
		The vertical one is natural in both $B$ and $X$ by assumption, while for the horizontal ones, naturality follows from the definition of $\cC_A$.
		
		Since identities in $\cC_A$ correspond to identities in $\cC$ with $A$ discarded, the infinite marginalization maps in $\cC_A$ are likewise given by discarding $A$ and applying the respective $\pi_F \in \cC(X^\N, X^F)$.
		In particular, the infinite marginalization maps in $\cC_A$ are deterministic, which means that the countable Kolmogorov power of $X$ in $\cC_A$ indeed exists and is given by $X^\N$.
	\end{proof}
	
	\begin{proof}[Proof of Theorem~\ref{syntheticthm}]
		We first consider the case $A = I$. Then we get
		\begin{equation}\label{eq:main_proof1}
			\begin{tikzpicture}
	\begin{pgfonlayer}{nodelayer}
		\node [style=none] (0) at (0, 0) {$=$};
		\node [style=state] (1) at (-2.5, -0.25) {$p$};
		\node [style=none] (2) at (-2.5, 1.5) {};
		\node [style=state] (13) at (4.5, -3) {$p$};
		\node [style=bn] (14) at (4.5, -1.75) {};
		\node [style=morphism] (15) at (2.75, 0) {$\tailcond{p}^\sharp$};
		\node [style=morphism] (16) at (6.25, 0) {$\tailcond{p}^\sharp$};
		\node [style=none] (17) at (2.75, -0.25) {};
		\node [style=none] (18) at (6.25, -0.25) {};
		\node [style=none] (19) at (2.75, 2.75) {};
		\node [style=none] (20) at (6.25, 2.75) {};
		\node [style=none] (21) at (4.5, 0.75) {${\cdots}$};
		\node [style=morphism] (22) at (2.75, 1.75) {$\samp$};
		\node [style=morphism] (23) at (6.25, 1.75) {$\samp$};
		\node [style=none] (24) at (8.75, 0) {$=$};
		\node [style=state] (25) at (13, -3) {$p$};
		\node [style=bn] (26) at (13, 0) {};
		\node [style=none] (29) at (11, 1.5) {};
		\node [style=none] (30) at (15, 1.5) {};
		\node [style=none] (31) at (11, 2.75) {};
		\node [style=none] (32) at (15, 2.75) {};
		\node [style=none] (33) at (13, 1.75) {${\cdots}$};
		\node [style=morphism] (34) at (11, 1.75) {$\samp$};
		\node [style=morphism] (35) at (15, 1.75) {$\samp$};
		\node [style=morphism] (36) at (13, -1.25) {$\tailcond{p}^\sharp$};
	\end{pgfonlayer}
	\begin{pgfonlayer}{edgelayer}
		\draw [style=double wire] (2.center) to (1);
		\draw [style=double wire] (13) to (14);
		\draw [style=double wire, in=-90, out=15] (14) to (18.center);
		\draw [style=double wire, in=-90, out=165] (14) to (17.center);
		\draw [style=double wire] (16) to (18.center);
		\draw [style=double wire] (17.center) to (15);
		\draw (19.center) to (22);
		\draw (22) to (15);
		\draw (20.center) to (23);
		\draw (23) to (16);
		\draw [in=-90, out=15] (26) to (30.center);
		\draw [in=-90, out=165] (26) to (29.center);
		\draw (31.center) to (34);
		\draw (32.center) to (35);
		\draw [style=double wire] (36) to (25);
		\draw (26) to (36);
	\end{pgfonlayer}
\end{tikzpicture}
		\end{equation}
		 as a consequence Lemma~\ref{lem:conditionally_iid}, together with the fact that for each morphism $f$, we have $f=\samp\circ f^\sharp$ (see Section~\ref{sec:Markov_cat}).
		 This already has the desired form of equation~\eqref{eq:de_finetti_claim}.
	
		Now, for a general morphism $p \colon A \to X^\N$, we apply this result to the parametric Markov category $\cC_A$, in which $p$ is represented by a morphism with domain $I$.
		By \cite[Lemmas 2.10 and 3.22]{fritz2020representable} and Lemma~\ref{lem:product_in_parametric}, $\cC_A$ satisfies our three assumptions provided that $\cC$ itself does.
		As is mentioned in \cite[Example 3.17]{fritz2020representable}, the sampling map in $\cC_A$ is represented by $\discard_{A} \otimes \, \samp$ in $\cC$.
		Therefore, instantiating equation \eqref{eq:main_proof1} in $\cC_A$ gives
		\begin{equation}\label{eq:main_proof2}
			\begin{tikzpicture}
	\begin{pgfonlayer}{nodelayer}
		\node [style=none] (0) at (0, 0) {$=$};
		\node [style=none] (2) at (-2.5, 1.5) {};
		\node [style=bn] (26) at (4.5, 1.25) {};
		\node [style=none] (29) at (2.5, 2.75) {};
		\node [style=none] (30) at (6.5, 2.75) {};
		\node [style=none] (31) at (2.5, 4) {};
		\node [style=none] (32) at (6.5, 4) {};
		\node [style=none] (33) at (4.5, 3) {${\cdots}$};
		\node [style=morphism] (34) at (2.5, 3) {$\samp$};
		\node [style=morphism] (35) at (6.5, 3) {$\samp$};
		\node [style=morphism] (36) at (4.5, 0) {$\;\; \tailcond{p}^\sharp \;\;$};
		\node [style=morphism] (37) at (-2.5, 0) {$p$};
		\node [style=none] (38) at (-2.5, -1.5) {};
		\node [style=morphism] (39) at (5.25, -1.75) {$p$};
		\node [style=bn] (40) at (4.5, -3) {};
		\node [style=none] (41) at (4.5, -4) {};
		\node [style=none] (42) at (3.75, 0) {};
		\node [style=none] (43) at (3.75, -1.75) {};
		\node [style=none] (44) at (5.25, 0) {};
		\node [style=none] (45) at (-2.5, -2) {$A$};
		\node [style=none] (46) at (-2.5, 2) {$X^\N$};
	\end{pgfonlayer}
	\begin{pgfonlayer}{edgelayer}
		\draw [in=-90, out=15] (26) to (30.center);
		\draw [in=-90, out=165] (26) to (29.center);
		\draw (31.center) to (34);
		\draw (32.center) to (35);
		\draw (26) to (36);
		\draw (38.center) to (37);
		\draw [style=double wire] (37) to (2.center);
		\draw [bend right] (40) to (39);
		\draw (41.center) to (40);
		\draw [in=-90, out=150] (40) to (43.center);
		\draw (43.center) to (42.center);
		\draw [style=double wire] (39) to (44.center);
	\end{pgfonlayer}
\end{tikzpicture}
		\end{equation}
		in $\cC$, which is the relevant form for what we wanted to show.
	\end{proof}

\smalltitle{Acknowledgments.}

	We want to thank the anonymous reviewer for the very helpful suggestions.
	Research for the second author is supported by NSERC Discovery grant RGPIN 2017-04383, and by the Perimeter Institute for Theoretical Physics.
	Research at Perimeter Institute is supported in part by the Government of Canada through the Department of Innovation, Science and Economic Development Canada and by the Province of Ontario through the Ministry of Colleges and Universities.
	Research for the third author is supported by the ERC grant ``BLaST --  Better Language for Statistics'', and by the University of Oxford.
 
\bibliographystyle{amsplain}

\providecommand{\bysame}{\leavevmode\hbox to3em{\hrulefill}\thinspace}

\end{document}